\documentclass[12pt]{amsart}
\usepackage{amscd,setspace,amssymb,amsopn,amsmath,amsthm,mathrsfs,lmodern,graphics,amsfonts,enumerate,verbatim,calc
} 
\usepackage[all]{xy}
\usepackage[colorlinks=true, citecolor=PineGreen, filecolor=PineGreen, linkcolor=Purple,pagebackref, hyperindex]{hyperref}
\usepackage[dvipsnames]{xcolor}
\usepackage{todonotes}
\usepackage{tikz-cd}
\usepackage{color, hyperref}
\usepackage[OT2,OT1]{fontenc}

\newcommand\cyr{%
\renewcommand\rmdefault{wncyr}%
\renewcommand\sfdefault{wncyss}%
\renewcommand\encodingdefault{OT2}%
\normalfont
\selectfont}
\DeclareTextFontCommand{\textcyr}{\cyr}
\usepackage{amssymb,amsmath}
\DeclareFontFamily{OT1}{rsfs}{}	
\DeclareFontShape{OT1}{rsfs}{n}{it}{<-> rsfs10}{}
\DeclareMathAlphabet{\fmathscr}{OT1}{rsfs}{n}{it}
\numberwithin{equation}{section}
\newtheorem{Theoremx}{Theorem}
 
\newtheorem{theorem}{Theorem}[section]
\newtheorem{lemma}[theorem]{Lemma}

\newtheorem{proposition}[theorem]{Proposition}
\newtheorem{corollary}[theorem]{Corollary}
\newtheorem{claim}[theorem]{Claim}

\newtheorem{conjecture}[theorem]{Conjecture}
\theoremstyle{definition}
\newtheorem{definition}[theorem]{Definition}
\newtheorem{remark}[theorem]{Remark}
\theoremstyle{remark}

\newcommand{\Ass}{\operatorname{Ass}}
\newcommand{\im}{\operatorname{Im}}
\renewcommand{\ker}{\operatorname{Ker}}

\newcommand{\Spec}{\operatorname{Spec}}

\newcommand{\Ht}{\operatorname{ht}}
\newcommand{\Height}{\operatorname{ht}}
\newcommand{\di}{\operatorname{div}}

\newcommand{\Mod}{\operatorname{Mod}}

\newcommand{\Ext}{\operatorname{Ext}}

\newcommand{\Hom}{\operatorname{Hom}}

\newcommand{\Gr}{\operatorname{Gr}}
\newcommand{\Ann}{\operatorname{Ann}}

\newcommand{\depth}{\operatorname{depth}}

\newcommand{\coker}{\operatorname{Coker}}

\newcommand{\height}{\operatorname{height}}

\newcommand{\lcb}{\operatorname{lcb}}

\newcommand{\N}{\mathbb{N}}

\newcommand{\Q}{\mathbb{Q}}

\newcommand{\fm}{\mathfrak{m}}
\newcommand{\fp}{\mathfrak{p}}

\newcommand{\fa}{\mathfrak{a}}
\newcommand{\fb}{\mathfrak{b}}

\newcommand{\fn}{\mathfrak{n}}

\newcommand{\NN}{\mathbb{N}}

\begin{document}
\title{On the equality of test ideals}

\author[Ian Aberbach]{Ian Aberbach}
\address{Department of Mathematics, University of Missouri, Columbia, MO 65211, USA}
\email{aberbachi@missouri.edu}

\author[Craig Huneke]{Craig Huneke}
\address{Department of Mathematics, University of Virginia, Charlottesville, VA 22903 USA}
\email{huneke@virginia.edu}

\author[Thomas Polstra]{Thomas Polstra}
\thanks{Polstra was supported in part by NSF Grant DMS \#2101890 and by a grant from the Simons Foundation, Grant Number 814268, MSRI}
\address{Department of Mathematics, University of Alabama, Tuscaloosa, AL 35401 USA}
\email{tmpolstra@ua.edu}


\begin{abstract}  We provide a natural criterion that implies equality of the test ideal and big test ideal in local rings of prime characteristic. Most notably, we show that the criterion is met by every local weakly $F$-regular ring whose anti-canonical algebra is Noetherian on the punctured spectrum.
\end{abstract}

 \maketitle


\section{Introduction}
Suppose $R$ is a Noetherian ring of prime characteristic $p>0$ and let $R^\circ$ be the set of elements which avoid all minimal primes of $R$. Let $I\subseteq R$ be an ideal of $R$ and denote by $I^{[p^e]}$ the expansion of $I$ along the $e$th iterate of the Frobenius endomorphism. The tight closure of $I$ is the ideal $I^*$ consisting of elements $x\in R$ such that there exists an element $c\in R^{\circ}$ with the property that $cx^{p^e}\in I^{[p^e]}$ for all $e\gg 0$. Unlike integral closure of ideals, the tight closure of an ideal does not commute with localization, \cite{BrennerMonsky}.  Brenner's and Monsky's counterexample to the localization problem leaves open the intriguing problem if the property of tight closure being a trivial operation on ideals commutes with localization.

Continue to let $R$ be a Noetherian ring of prime characteristic $p>0$. The ring $R$ is called \emph{weakly $F$-regular} if every ideal is tight closed, that is $I=I^*$ for every ideal $I$.\footnote{A defining property of tight closure theory is that every regular ring is weakly $F$-regular.} A ring is called \emph{$F$-regular} if every localization of $R$ is weakly $F$-regular. Let $F^e_*R$ denote the restriction of scalars of $R$ along the $e$th iterate Frobenius endomorphism $F^e:R\to R$. We say that $R$ is \emph{strongly $F$-regular} if for each $c\in R^\circ$ there exists $e\in \N$ such that the $R$-linear map $R\to  F^e_*R$ defined by $1\mapsto F^e_*c$ is pure. Every strongly $F$-regular ring is weakly $F$-regular and the property of being strongly $F$-regular passes to localization. It is conjectured that all three notions of $F$-regularity agree. 

\begin{conjecture}[The weak implies strong conjecture]
\label{Conjecture Weak implies strong}
If $R$ is an excellent weakly $F$-regular ring of prime characteristic $p>0$ then $R$ is strongly $F$-regular.
\end{conjecture}

Williams proved Conjecture~\ref{Conjecture Weak implies strong} for the class of $3$-dimensional rings, \cite{Williams}.  Every excellent $4$-dimensional $F$-regular ring is strongly $F$-regular by pairing \cite[Corollary~4.4]{AberbachPolstraJOA} with \cite[Corollary~K]{MixedCharMMPPaper}. The purpose of this article is to extend the results of \cite{AberbachPolstraJOA} to rings of arbitrary dimension. In particular, if the results of the prime characteristic minimal model program in dimension $3$ established in \cite{MixedCharMMPPaper} are valid in all dimensions, then the classes of excellent $F$-regular and excellent strongly $F$-regular rings are equivalent.

A prime characteristic ring $R$ is weakly $F$-regular if and only if $R_\fm$ is weakly $F$-regular for every maximal ideal $\fm\in \Spec(R)$. Moreover, an excellent local ring is weakly $F$-regular if and only if its completion is weakly $F$-regular. Every local weakly $F$-regular ring is a Cohen-Macaulay normal domain. We therefore restrict our attention to the class of local Cohen-Macaulay normal domains which admit a canonical module.


\begin{Theoremx}
\label{Main Theorem 1}
Let $(R,\fm,k)$ be an excellent Cohen-Macaulay normal domain of prime characteristic $p>0$, of Krull dimension $d$, and $I\subseteq R$ an anti-canonical ideal.\footnote{An ideal $I\subseteq R$ is an \emph{anti-canonical ideal} if it represents the inverse of the canonical divisor in the class group of $R$. Equivalently, there exists a canonical ideal $J\subseteq R$, with components disjoint from that of $I$, so that $I\cap J$ is a principal ideal.} Suppose that there exists an $m\in\NN$ such that $I^{(m)}$ is principal when localized at a each height $2$ prime\footnote{Every excellent normal ring which is $F$-rational in codimension $2$ admits an $m\geq 1$ with this property. Indeed, $F$-rational rings have pseudo-rational singularities, excellent pseudo-rational singularities are rational in codimension $2$, and $2$-dimensional excellent local rational singularities have torsion class group, \cite{SmithRationalSingularities, LipmanResolution, LipmanRationalSingularities}.} and for each $1\leq j\leq d-2$ there exists an ideal $\fa_j$ of height $d-j+1$ such that
\[
\fa_j^{p^e}H^{j}_\fm\left(\frac{R}{I^{(mp^e)}}\right)=0
\]
for every $e\in\NN$. If $R$ is weakly $F$-regular then $R$ is strongly $F$-regular.
\end{Theoremx}

\begin{remark}
The Matlis dual of the local cohomology module $H^{j}_\fm(R/I^{(mp^e)})$ is the completion of $\Ext^{d-j}_R(R/I^{(mp^e)},J)$, a module which is not supported in codimension $d-j$ if $j\leq d-2$. Hence $H^{j}_\fm(R/I^{(mp^e)})$ is annihilated by an ideal of height $d-j+1$. The criterion of Theorem~\ref{Main Theorem 1} is therefore reasonable as it is natural to anticipate that the annihilators of $H^{j}_\fm(R/I^{(mp^e)})$ are of linear comparisons as $e\to \infty$.
\end{remark}

\begin{remark}
    Let $(R,\fm,k)$ be as in Theorem~\ref{Main Theorem 1} and assume that $R$ is Cohen-Macaulay. Suppose that $E_R(k)$ is an injective hull of the residue field. Let $0^*_{E_R(k)}$ and $0^{*fg}_{E_R(k)}$ denote the tight closure and finitisitic tight closure respectively of the $0$-submodule of $E_R(k)$, see Section~\ref{lcb introduction} for definitions.  Then $0^{*fg}_{E_R(k)}=0^*_{E_{R(k)}}$ under the hypotheses of Theorem~\ref{Main Theorem 1}, see Theorem~\ref{Main Theorem 1 again}. Therefore the test ideal and big test ideal of $R$ agree by \cite[Proposition~8.23]{HHJAMS} and \cite[Theorem~3.2]{AberbachEnescuTestIdeals}, c.f. \cite[Theorem~7.1 and Theorem~7.2]{LyubeznikSmithTestIdeal}. By definition, the test ideal of $R$ is the unit ideal if and only if $R$ is weakly $F$-regular and the big test ideal of $R$ is the unit ideal if and only if $R$ is strongly $F$-regular.  Therefore Theorem~\ref{Main Theorem 1} is a consequence of Theorem~\ref{Main Theorem 1 again}.
\end{remark}

Conjecture~\ref{Conjecture Weak implies strong} is valid for rings $R$ which are standard graded over a field, \cite{LyubeznikSmith}. It would be interesting to know if such rings satisfy the hypotheses of Theorem~\ref{Main Theorem 1}. Without the standard graded assumption, most established cases of Conjecture~\ref{Conjecture Weak implies strong} require an assumption on $R$ that is akin to being Gorenstein. Hochster and Huneke proved Conjecture~\ref{Conjecture Weak implies strong} for the class of Gorenstein rings, \cite{HHJAG}. Building upon Williams' proof of Conjecture~\ref{Conjecture Weak implies strong} for the class of $3$-dimensional rings, \cite{Williams}, MacCrimmon proved the weak implies strong conjecture for rings which are $\mathbb{Q}$-Gorenstein on the punctured spectrum, \cite{Maccrimmon}.  Singh announced that Conjecture~\ref{Conjecture Weak implies strong} is valid for rings whose anti-canonical algebra\footnote{Suppose that $R$ is a normal domain and $I\subseteq R$ is an anti-canonical ideal. The \emph{anti-canonical algebra of $R$} is the symbolic Rees algebra $R\oplus I\oplus I^{(2)}\oplus \cdots$, an algebra unique up to $R$-algebra isomorphism.} is Noetherian. Singh's result was never published, but has since been recaptured by others, \cite{ChiecchioEnescuMillerSchwede}.

Singularities of prime characteristic rings are related to KLT singularities of the complex minimal model program through the process of reduction to prime characteristic, \cite{HaraWatanabe, Takagi}. Theorems of the complex minimal model program establish that if $R$ is essentially of finite type over $\mathbb{C}$ with at worst KLT singularities, then the symbolic Rees algebras associated to ideals of pure height $1$ are Noetherian.  It is therefore natural to conjecture the same in strongly $F$-regular rings and that the hypotheses of Singh's Theorem are vacuous. 

\begin{conjecture}
\label{Conjecture SFR implies finitely generated divisorial blowups}
If $R$ is an excellent strongly $F$-regular ring of prime characteristic $p>0$ and $I\subseteq R$ an ideal of pure height $1$. Then the symbolic Rees algebra of $I$ is Noetherian.
\end{conjecture}

Progress around Conjecture~\ref{Conjecture SFR implies finitely generated divisorial blowups} is quite limited. An elementary and (mostly) algebraic proof of Conjecture~\ref{Conjecture SFR implies finitely generated divisorial blowups} for the class of $2$-dimensional $F$-regular rings can be derived from \cite[Corollary~3.2]{PolstraMCM}. Recent progress of the minimal model program establishes Conjecture~\ref{Conjecture SFR implies finitely generated divisorial blowups} for the class of $3$-dimensional $F$-regular rings, see \cite[Corollary~K]{MixedCharMMPPaper} and \cite[Proof of Corollary~4.5]{AberbachPolstraJOA} for necessary details.

In light of Conjecture~\ref{Conjecture SFR implies finitely generated divisorial blowups}, it would be desirable to remove the assumption that the anti-canonical algebra of $R$ is Noetherian in Singh's Theorem and replace it with the milder hypothesis that the anti-canonical algebra is assumed to be Noetherian at non-closed points of $\Spec(R)$. Such a step puts forth a much needed inductive program to establish Conjecture~\ref{Conjecture SFR implies finitely generated divisorial blowups}, or at the very least establish that the class of $F$-regular and strongly $F$-regular rings agree. This is what we accomplish and is the main contribution of this article.

\begin{Theoremx}
\label{Main Theorem 2}
Let $(R,\fm,k)$ be an excellent weakly $F$-regular ring of prime characteristic $p>0$, of Krull dimension $d$, and $I\subseteq R$ an anti-canonical ideal. Suppose that the anti-canonical algebra of $R$ is Noetherian on the punctured spectrum. There exists $m\in\NN$ so that $I^{(m)}$ is principal when localized at each height $2$ prime and for each $1\leq j\leq d-2$ there exists an ideal $\fa_j$ of height $d-j+1$ such that
\[
\fa_j^{p^e}H^{j}_\fm\left(\frac{R}{I^{(mp^e)}}\right)=0
\]
for every $e\in\NN$. In particular, the ring $R$ is strongly $F$-regular by Theorem~\ref{Main Theorem 1}.
\end{Theoremx}

\begin{remark}
    The implications of the techniques employed in this article regarding the agreement between the test ideal and big test ideal of $R$ are not explicitly clear when only considering the assumption that the anti-canonical algebra is Noetherian on the punctured spectrum. Our approach requires not only Noetherianity of the anti-canonical algebra of $R$ on the punctured spectrum, but also the additional condition of Cohen-Macaulayness on the punctured spectrum. We observe that this condition holds true if $R$ is weakly $F$-regular, see the proof of Corollary~\ref{Corollary: annihilation of local cohomology in splinter local rings}. To establish the equality of test ideals solely based on the assumption that the anti-canonical algebra is Noetherian on the punctured spectrum, one would need to appropriately modify the outcomes and methodologies presented in Section~\ref{Section Local Cohomology} to accommodate algebras that may not be Cohen-Macaulay.
\end{remark}
\section{Annihilators of Local Cohomology}
\label{Section Local Cohomology}

This section is devoted to proving Theorem~\ref{Main Theorem 2}. Let $(R,\fm,k)$ be an excellent local normal domain of Krull dimension $d\geq 3$ and $I\subseteq R$ an ideal of pure height $1$. Let $W=R\setminus \bigcup_{P\in\min(I)}P$ and for each $n\in\NN$ let $I^{(n)}=I^nR_W\cap R$ denote the $n$th symbolic power of the ideal $I$. To study the annihilators of $H^i_\fm(R/I^{(n)})$ we will approximate the ideals $I^{(n)}$ by ideals of the form $\overline{(y_1,\ldots,y_h)^n}$ where $h$ is ``small,'' $y_1,\ldots,y_h\in I$, and $\overline{J}$ denotes the integral closure of an ideal $J\subseteq R$.

Let $J\subseteq R$ be an ideal and $n\in\NN$. There are short exact sequences
\[
0\to \frac{\overline{J^{n-1}}}{\overline{J^{n}}}\to \frac{R}{\overline{J^{n}}}\to \frac{R}{\overline{J^{n-1}}}\to 0,
\]
and so there are exact sequences of local cohomology modules
\[
H^i_\fm\left(\frac{\overline{J^{n-1}}}{\overline{J^{n}}}\right)\to H^i_\fm\left(\frac{R}{\overline{J^{n}}}\right)\to H^i_\fm\left(\frac{R}{\overline{J^{n-1}}}\right).
\]
Our aim is to establish uniform annihilators of the local cohomology modules $H^i_\fm(\overline{J^{n-1}}/\overline{J^n})$ that are independent of $n$.  For the sake of convenience, we adopt the following notation: 
\begin{itemize}
    \item $R[Jt]=\bigoplus_{n\geq 0} J^nt^n$ is the Rees algebra of $J$;
    \item $R[Jt,t^{-1}]=\bigoplus_{n\in \NN} J^nt^n$ is the extended Rees algebra of $J$, i.e. $R[Jt,t^{-1}]$ agrees with the Rees algebra $R[Jt]$ in positive degree and contains copies of $R$ in negative degree;
    \item $\mathcal{R}$ is the integral closure of the Rees algebra $R[Jt]$ in $R[t]$; $\mathcal{R}$ is $\NN$-graded and the $n$th graded piece of $\mathcal{R}$ is $\overline{J^n}$;
    \item $\mathcal{R}[t^{-1}]$ is the integral closure of the extended Rees algebra $R[Jt,t^{-1}]$ in $R[t,t^{-1}]$. If $n\geq 0$ then the $n$th graded piece of $\mathcal{R}[t^{-1}]$ is $\overline{J^n}$. The algebra $\mathcal{R}[t^{-1}]$ contains copies of $R$ in negative degrees.
\end{itemize}

If $x\in R$ then $xH^i_\fm(\overline{J^{n-1}}/\overline{J^n})=0$ for all $n\in\NN$ if and only if 
\[
xH^i_\fm(\mathcal{R}[t^{-1}]/(t^{-1}\mathcal{R}[t^{-1}])=0.
\]
The Faltings Annihilator Theorem, later generalized by Brodmann, provides a criterion to establish such annihilation properties.

\begin{theorem}[{\cite[Theorem~9.5.1]{BrodmannSharpBook}}]
\label{Thm:Brodmann-Faltings}
Let $S$ be a Noetherian ring which is the homomorphic image of a regular ring, $M$ a finitely generated $S$-module, and let $\fa,\fb\subseteq R$ be ideals. Then
\[
\min\{i\in\mathbb{N}\mid \not\exists C : \fa^{C}H^i_\fb(M)=0\}=\min\{\depth(M_P)+\height((\fb+P)/P)\mid P\not\in V(\fa)\}.
\]
\end{theorem}

Our first step towards proving Theorem~\ref{Main Theorem 2} is the following lemma.

\begin{lemma}
\label{Lemma: Uniform annihilator of local cohomology}
Let $(R,\fm,k)$ be an excellent local normal domain of Krull dimension $d$ and $J\subseteq R$ an ideal generated by at most $h$ elements.  Suppose that the associated graded algebra $\bigoplus_{n\geq 0} \overline{J^n}/\overline{J^{n+1}}\otimes_R R_x$ is Cohen-Macaulay. Then there exists a constant $C$ so that 
\[
x^{C}H^i_\fm(\overline{J^n}/\overline{J^{n+1}})=0
\]
for every $0\leq i\leq d-h-1$ and $n\in\mathbb{N}$.
\end{lemma}

\begin{proof}
Without loss of generality, we may pass to the completion of $R$ and assume that $R$ is the homomorphic image of a regular local ring. Let $S=\mathcal{R}[t^{-1}]$ and $G=S/t^{-1}S$. The lemma is equivalent to the assertion that there exists a constant $C$ so that $x^CH^i_{\fm S}(G)=0$ for every $1\leq i\leq d-h-1$. By Theorem~\ref{Thm:Brodmann-Faltings}, it suffices to show that if $P\in \Spec(S)\setminus V(xS)$ then 
\[
\depth(G_P)+\height\left(\frac{\fm S+P}{P}\right)\geq d-h.
\]
If $P\not\in V(xS)$ then $G_P$ is Cohen-Macaulay. Therefore
\[
\depth(G_P)=\dim(G_P)=\height_S(P)-1.
\]
Then, because $S$ is catenary,
\begin{align*}
\depth(G_P)+\height\left(\frac{\fm S+P}{P}\right) &=\height_S(P)-1+\dim(S/P)-\dim(S/\fm S+P)\\
&=\height_S(P)-1+d+1-\height_S(P)-\dim(S/\fm S+P)\\
&=d-\dim(S/\fm S+P).
\end{align*}
Recall that $S$ is the integral closure of $\mathcal{R}[t^{-1}]$ in $R[t,t^{-1}]$. It follows that $S/\fm S$ is a finite extension of the fiber cone of $J$, an $R/\fm$-algebra of Krull dimension at most $h$. Therefore 
\[
\dim(S/\fm S+P)\leq \dim(S/\fm S)= h
\]
and so $\depth(G_P)+\height\left(\frac{\fm S+P}{P}\right)\geq d-h$ as needed.
\end{proof}

\begin{corollary}
\label{Corollary: linear annihilation of lc ideal generated by small number of elements}
Let $(R,\fm,k)$ be an excellent local Cohen-Macaulay normal domain of Krull dimension $d$ and $J\subseteq R$ an ideal generated by at most $h$ elements. Suppose that the ring $\bigoplus \overline{J^n}/\overline{J^{n+1}}\otimes_R R_x$ is Cohen-Macaulay. Then there exists a constant $C$ so that 
\[
x^{Cn}H^i_\fm(R/\overline{J^{n}})=0
\]
for every $0\leq i\leq d-h-1$ and $n\in\NN$.
\end{corollary}

\begin{proof}
For every $i\geq 0$ and for every $n\in \NN$ there are exact sequences of local cohomology modules
\[
H^i_\fm(\overline{J^n}/\overline{J^{n+1}})\to H^i_\fm(R/\overline{J^{n+1}})\to H^i_\fm(R/\overline{J^n}).
\]
By Lemma~\ref{Lemma: Uniform annihilator of local cohomology}, if $i\leq d-h-1$, then there exists a constant $C$ so that $x^C$ annihilates the left most module of the above exact sequences for all $n\geq 0$. By induction, $x^{Cn}$ annihilates $H^i_\fm(R/\overline{J^n})$ for every $n\in\NN$.
\end{proof}

\begin{remark}
If we are only interested in annihilation properties of $H^1_\fm(\overline{J^n}/\overline{J^{n+1}})$, then many of the assumptions of Lemma~\ref{Lemma: Uniform annihilator of local cohomology} and Corollary~\ref{Corollary: linear annihilation of lc ideal generated by small number of elements} are not necessary. One only needs to assume that $R$ is an excellent normal domain and $J$ is generated by at most $d-2$ elements to conclude that there exists a constant $C$ so that $\fm^C$ annihilates $H^1_\fm(\overline{J^n}/\overline{J^{n+1}})$ for every $n\in\NN$. Indeed, $\height(\fm S+P/P)\geq 1$ for all $P\in \Spec(S)\setminus V(\fm S)$. Thus, to show
\[
\depth(G_P)+\height(\fm S +P/P)\geq 2
\]
for every $P\in \Spec(S)\setminus V(\fm S)$, it suffices to show that $\height(\fm S+P/P)\geq 2$ whenever $\depth(G_P)=0$. If $\depth(G_P)=0$ then $P\in \Spec(S)$ is an associated prime of $t^{-1}S$. The ring $S$ is normal and $t^{-1}$ is a nonzerodivisor. Therefore every associated prime of $t^{-1}S$ is minimal and so $\dim(G_P)=0$. One can now proceed as in the proof of Lemma~\ref{Lemma: Uniform annihilator of local cohomology} to show that $\height(\fm S+P/P)\geq 2$.
\end{remark}

\begin{lemma}
\label{lemma: approximating integral closures by generic reductions}
Let $(R,\fm,k)$ be an excellent Noetherian local normal domain with infinite residue field, $I\subseteq R$ an ideal, $P_1,\ldots,P_t\in \Spec(R)$ a finite collection of non-comparable prime ideals, and $W=R\setminus \bigcup_{i=1}^t P_i$. Suppose that $\ell_{R_{P_i}}(IR_{P_i})\leq h$ for every $1\leq i\leq t$. Then there exist elements $y_1,\ldots,y_h\in I$ and $x\in W$ with the following properties:
\begin{enumerate}
    \item $(y_1,\ldots,y_h)R_W\subseteq IR_W$ is a reduction of $IR_W$;
    \item $x^n\overline{I^n}\subseteq \overline{(y_1,\ldots,y_h)^n}$ for all $n\in\mathbb{N}$.
\end{enumerate}
\end{lemma}

\begin{proof}
Recall the following: Suppose $(S,\fn,\ell)$ is a local ring and $J\subseteq I$ are ideals. Then $J$ is reduction of $I$ if and only if $S[Jt]\otimes_S \ell\to S[It]\otimes_S \ell$ is finite, see \cite[Proposition~8.2.4]{SwansonHuneke}. In particular, if $J'\subseteq I$ is an ideal such that $J'\equiv J+\fn I$ and $J$ is a reduction of $I$ then $J'$ is a reduction of $I$.

To prove the lemma start by choosing elements $y_{1,i},\ldots,y_{h,i}\in I$ so that $(y_{1,i},\ldots,y_{h,i})R_{P_i}$ forms a reduction of the ideal $IR_{P_i}$. Choose elements $r_j\in (\cap_{i\not = j}P_i)\setminus P_j$ and set $y_i=\sum r_jy_{j,i}$. Then $(y_1,\ldots,y_h)R_{P_i}$ forms a reduction of $IR_{P_i}$ for each $1\leq i\leq t$ by the above discussion. Therefore $(y_1,\ldots,y_h)R_W$ forms a reduction of $IR_W$ by \cite[Propositions 8.1.1.]{SwansonHuneke}.

Let $J=(y_1,\ldots,y_h)$. Then $\overline{J}R_W=\overline{I}R_W$ and so there exists an element $x\in W$ such that $x\overline{I}\subseteq \overline{J}$, in particular $xI\subseteq \overline{J}$. Raising the containment to the $n$th power we find that $x^nI^n\subseteq \overline{J^n}$ for every $n\in\NN$. We claim that $x^n\overline{I^n}\subseteq \overline{J^n}$. Let $r\in \overline{I^n}$, then there exists a $t\in\NN$ and an equation
\[
r^t+a_1r^{t-1}+\cdots+ a_{t-1}r+a_t=0
\]
such that $a_j\in I^{nj}$ for each $1\leq j\leq t$. Multiplying by $x^{nt}$ we find that there is an equation
\[
(x^nr)^t+x^na_1(x^nr)^{t-1}+\cdots x^{n(t-1)}a_{t-1}+x^{nt}a_t=0.
\]
The elements $x^{nj}a_j$ belong to $\overline{J^{nj}}$ and therefore $x^nr\in \overline{\overline{J^n}}=\overline{J^n}$.
\end{proof}

\begin{theorem}
\label{Theorem: linear annihilation of all local cohomology}
Let $(R,\fm,k)$ be an excellent local Cohen-Macaulay normal domain of Krull dimension $d\geq 3$ and $I\subseteq R$ an ideal of pure height $1$ with the following properties:
\begin{itemize}
    \item $I^nR_P=I^{(n)}R_P$ for every $P\in \Spec(R)\setminus\{\fm\}$ and every $n\in\mathbb{N}$;
    \item If $P\in \Spec(R)\setminus\{\fm\}$ and $G=\mathcal{R}[t^{-1}]/t^{-1}\mathcal{R}[t^{-1}]$ is the associated graded ring of $I$ then $G_P$ is Cohen-Macaulay.
\end{itemize}
Then there exists a system of parameters $x_1,x_2,\ldots,x_d$ such that for every $3\leq t\leq d$;
\[
(x_1^n,\ldots,x_{t}^n)H^{j}_\fm(R/I^{(n)})=0
\]
for every $0\leq j\leq d-(t-1)$ and $n\in\NN$.
\end{theorem}

\begin{proof}
The ideal $I$ is locally principal at its height $1$ components because $R$ is normal. Therefore $\overline{I^n}\subseteq I^{(n)}$. Our assumptions inform us that $I^{(n)}/\overline{I^n}$ is $0$-dimensional for every integer $n$. Therefore for every integer $i\geq 1$
\[
H^i_\fm(R/\overline{I^{n}})\cong H^i_\fm(R/I^{(n)}).
\]
Start by choosing $x_1\in I$. Then clearly $x_1^n\in \overline{I^{n}}$ and therefore $x_1^n$ annihilates $H^i_\fm(R/\overline{I^{n}})$ for all integers $i$ and $n$. If $W_1$ is the complement of the union of the minimal primes of $x_1R$ then $IR_{W_1}$ is a principal ideal. By Lemma~\ref{lemma: approximating integral closures by generic reductions} there exists an element $y\in I$ and $x\in W_1$ so that $x^n\overline{I^n}\subseteq y^nR$ for every $n\in\NN$. There are short exact sequences
\[
0\to \frac{\overline{I^n}}{y^nR}\to \frac{R}{y^nR}\to \frac{R}{\overline{I^n}}\to0
\]
and so $H^j_\fm(R/\overline{I^n})\cong H^{j+1}_\fm(\overline{I^n}/y^nR)$ if $j\leq d-3$ and there is an injective map $H^{d-2}_\fm(R/\overline{I^n})\to H^{d-1}_\fm(\overline{I^n}/y^nR)$. Therefore $x^n$ annihilates $H^j_\fm(R/\overline{I^n})$ for every $j\leq d-2$ and we take $x_2=x$.

If $W_2$ is the complement of the union of the minimal primes of $(x_1,x_2)$ then $IR_{W_2}$ has analytic spread at most $1$, see \cite[Proof of Theorem~1.5]{CutkoskyHerzogSrinivasan}. The ring $R_{W_2}$ is normal, every principal ideal in a normal ring is integrally closed, and therefore $IR_{W_2}$ is a principal ideal. We therefore proceed as before to find an element $x_3$ so that $x_3^n$ annihilates $H^j_\fm(R/\overline{I^n})$ for every $j\leq d-2$ as needed.

Inductively, suppose that we have found parameter elements $x_1,\ldots,x_i$, with $i\geq 3$, so that if $3\leq t\leq i$ then
\[
(x_1^n,\ldots,x_t^n)H^{j}_\fm(R/\overline{I^{n}})=0
\]
for every $0\leq j\leq d-(t-1)$. It is important that $t\geq 3$ in the inductive step of the proof. If $t=2$ then it is not the case that $(x_1^n,x_2^n)$ annihilates $H^j_\fm(R/\overline{I^n})$ for every $0\leq j\leq d-(2-1)=d-1$. Indeed, the annihilator of the top local cohomology module $H^{d-1}_\fm(R/\overline{I^n})$ is the height $1$ ideal $I^{(n)}$ and $(x_1^n,x_2^n)\not\subseteq I^{(n)}$.
If $i=d$ then we are done. Suppose that $i\leq d-1$. Our aim is to find a parameter element $x_{i+1}$ so that 
\[
x_{i+1}^nH^{j}_\fm(R/\overline{I^{n}})=0
\]
for every $j\leq d-i$.

Let $W$ be the complement of the union of the minimal primes of the parameter ideal $(x_1,\ldots,x_i)$. Then $I^nR_W=I^{(n)}R_W$ for all integers $n$ and so the localization of $IR_W$ at a maximal ideal of $R_W$ is an ideal of analytic spread at most $i-1$, see \cite[Proof of Theorem~1.5]{CutkoskyHerzogSrinivasan}. By Lemma~\ref{lemma: approximating integral closures by generic reductions} there exists elements $y_1,\ldots,y_{i-1}\in I$ and $x_{i+1}'\in W$ so that  
\begin{enumerate}
    \item $(y_1,\ldots,y_{i-1})R_W\subseteq IR_W$ is a reduction of $IR_W$;
    \item $(x_{i+1}')^n\overline{I^n}\subseteq \overline{(y_1,\ldots,y_{i-1})^n}$ for all $n\in\mathbb{N}$.
\end{enumerate}

Let $J=(y_1,\ldots,y_{i-1})$ and consider the short exact sequences
\[
0\to \frac{\overline{I^n}}{\overline{J^n}}\to \frac{R}{\overline{J^n}}\to \frac{R}{\overline{I^n}}\to 0.
\]
The element $(x_{i+1}')^n$ annihilates the left-most module in the above short exact sequence and there are exact sequences of local cohomology modules
\[
H^j_\fm\left(\frac{R}{\overline{J^n}}\right)\to H^j_\fm\left(\frac{R}{\overline{I^n}}\right)\to H^{j+1}_\fm\left(\frac{\overline{I^n}}{\overline{J^n}}\right).
\]
The element $(x_{i+1}')^n$ annihilates the right-most module. By our hypothesis that the associated graded ring of $I$ is Cohen-Macaulay on the punctured spectrum of $R$, Corollary~\ref{Corollary: linear annihilation of lc ideal generated by small number of elements} implies that there exists a constant $C$ so that $(x_{i+1}')^{Cn}$ annihilates $H^j_\fm(R/\overline{J^n})$ for every $j\leq d-(i-1)-1=d-i$. Therefore $(x_{i+1}')^{C(n+1)}$ annihilates $H^j_\fm(R/\overline{I^n})$ for every $j\leq d-i$. Therefore $x_{i+1}=(x_{i+1}')^{2C}$ has the desired annihilation properties. 
\end{proof}

Theorem~\ref{Main Theorem 2} is a consequence of the following theorem.

\begin{corollary}
\label{Corollary: annihilation of local cohomology in splinter local rings}
Let $(R,\fm,k)$ be an excellent local Cohen-Macaulay normal domain of prime characteristic $p>0$ and Krull dimension $d\geq 3$. Suppose that $R$ is a splinter on the punctured spectrum of $R$ and that the anti-canonical algebra of $R$ is Noetherian on the punctured spectrum. Then there exists an ideal $I\subseteq R$ of pure height $1$ and parameters $x_1,\ldots,x_d$ with the following properties:
\begin{enumerate}
    \item $I\cong \omega_R^{(-m)}$ for some $m\geq 1$;
    \item For each $1\leq j\leq d-2$, the ideal $\fa_j:=(x_1,\ldots,x_{d-j+1})$ is such that
    \[
    \fa_j^{[p^e]}H^j_\fm(R/I^{(p^e)})=0
    \] 
    for each $e\in\NN$.
    \end{enumerate}
\end{corollary}

\begin{proof}
Start by choosing an ideal $I\subseteq R$ of pure height $1$ so that $I\cong \omega_R^{(-1)}$ is an anti-canonical ideal. We are assuming that the anti-canonical algebra is Noetherian on the punctured spectrum. Therefore if $P\in\Spec(R)\setminus\{\fm\}$ then there exists an integer $m$ such that the symbolic Rees algebra of $I^{(m)}R_P$ is standard graded. The set of prime ideals $\cup \Ass(R/I^{(m)n})$ is a finite set by \cite{Brodmann}, see also \cite{HunekeSmrinov}. By prime avoidance there exists an $s\in R\setminus P$ which is contained in each non-minimal member of $\cup \Ass(R/I^{(m)n})$. Then $I^{(m)n}R_s=I^{(mn)}R_s$ for every $n\in\NN$. The space $\Spec(R\setminus\{\fm\}$ is quasi-compact. Therefore there exists finitely many open sets $D(s_1),\ldots,D(S_t)$ covering $\Spec(R\setminus\{\fm\}$ and integers $m_i$, $1\leq i\leq s$, so that for all $n\in \NN$ $I^{(m_i)n}R_{s_i}=I^{(m_in)}R_{s_i}$. If $m$ is a common multiple of $m_1,\ldots,m_s$ and then the symbolic Rees algebra of $I^{(m)}$ is standard graded on the punctured spectrum, i.e. $I^{(m)n}R_P=I^{(mn)}R_P$ for all $n\in \NN$ and $P\in\Spec(R)\setminus\{\fm\}$. We replace $I$ by $I^{(m)}$. By Theorem~\ref{Theorem: linear annihilation of all local cohomology}, it suffices to show that $R$ is Cohen-Macaulay and that if $G=\Gr_I(R)$ is the associated graded ring of $I$ then $G_P$ is a Cohen-Macaulay algebra for all $P\in\Spec(R)\setminus\{\fm\}$.

For each non-maximal prime $P$ the localized ring $R_P$ is strongly $F$-regular by \cite[Corollary~5.9]{ChiecchioEnescuMillerSchwede}, see also \cite[Theorem~0.1]{WatanabeInfinite}. Therefore the localized (symbolic) Rees algebras $R[It]\otimes R_P$ are Cohen-Macaulay for all $P\in\Spec(R)\setminus\{\fm\}$, see \cite[Lemma~6.1]{ChiecchioEnescuMillerSchwede}. We may now conclude that $G_P$ is a Cohen-Macaulay algebra for all $P\in\Spec(R)\setminus\{\fm\}$ by \cite[Proposition~1.1]{HunekeAssociatedGraded}.
\end{proof}

\section{Tight closure, local cohomology, and local cohomology bounds}\label{lcb introduction}

\subsection{Tight closure}
 Let $R$ be a ring of prime characteristic $p>0$ and let $R^\circ$ be the complement of the union of the minimal primes of $R$. The $e$th Frobenius functor, or the $e$th Peskine-Szpiro functor, is the functor $F^e:\Mod(R)\to \Mod(R)$ obtained by extending scalars along the $e$th iterate of the Frobenius endomorphism. If $N\subseteq M$ are $R$-modules and $m\in M$, then $m$ is in the tight closure of $N$ relative to $M$ if there exists a $c\in R^\circ$ such that for all $e\gg0$ the element $m$ is in the kernel of the following composition of maps:
  \[
 M\to M/N\to F^e(M/N)\xrightarrow{\cdot c} F^e(M/N).
 \]
 In particular, if we consider an inclusion of $R$-modules of the form $I\subseteq R$ then $F^e(R/I)\cong R/I^{[p^e]}$ where $I^{[p^e]}=(r^{p^e}\mid r\in I)$, and an element $r\in R$ is in the tight closure of $I$ relative to $R$ if there exists a $c\in R^\circ$ such that $cr^{p^e}\in I^{[p^e]}$ for all $e\gg 0$.  The tight closure of the module $N$ relative to the module $M$ is denoted $N^*_{M}$. In the case that $M=R$ and $N=I$ is an ideal then we denote the tight closure of $I$ relative to $R$ as $I^*$. We say that $N$ is tightly closed in $M$ if $N=N^*_M$. If an ideal is tightly closed in $R$ then we simply say that the ideal is tightly closed. The finitistic tight closure of $N\subseteq M$ is denoted $N^{*fg}_M$ and is the union of $(N\cap M')^*_{M'}$ where $M'$ runs over all finitely generated submodules of $M$.

The notions of weak $F$-regularity and strong $F$-regularity can be compared by studying the finitistic tight closure and tight closure of the zero submodule of the injective hull of a local ring by \cite[Proposition~8.23]{HHJAMS} and \cite[Proposition~7.1.2]{SmithThesis}. Suppose that $(R,\fm,k)$ is complete local and $E_R(k)$ is the injective hull of the residue field. The finitistic test ideal of $R$ is $\tau_{fg}(R)=\bigcap_{I\subseteq R} \Ann_R(I^*/I)$ and agrees with $\Ann_R(0^{*fg}_{E_R(k)})$. The (big) test ideal of $R$ is $\tau(R)=\bigcap_{N\subseteq M\in \Mod(R)}\Ann_R(N^*_M/N)$ and agrees with $\Ann_R(0^*_{E_R(k)})$. The ring $R$ is weakly $F$-regular if and only if $\tau_{fg}(R)=R$ and $R$ is strongly $F$-regular if and only if $\tau(R)=R$. Thus to prove the conjectured equivalence of weak and strong $F$-regularity it is enough to show $0^*_{E_R(k)}=0^{*fg}_{E_R(k)}$ under hypotheses satisfied by rings which are weakly $F$-regular.

To explore the tight closure of the zero submodule of $E_R(k)$ we exploit the structure of $E_R(k)$ as a direct limit of $0$-dimensional Gorenstein quotients of $R$ described in \cite{HochsterPurity}. Suppose $(R,\fm,k)$ is a complete local Cohen-Macaulay domain of Krull dimension $d$ and $J_1\subsetneq R$ a canonical ideal. Let $0\neq x_1\in J_1$, $x_2,\ldots,x_d\in R$ a parameter sequence, and for each $t\in \N$ let $I_t=(x_1^{t-1}J_1,x_2^t,\ldots,x_d^t)$. The sequences of ideals $\{I_t\}$ form a decreasing sequence of irreducible $\fm$-primary ideals cofinal with $\{\fm^t\}$. Moreover, the direct limit  system $\varinjlim R/I_t\xrightarrow{\cdot x_1\cdots x_d}R/I_{t+1}$ is isomorphic to $E_R(k)$. The following lemma uses this description of the injective hull of the residue field to describe any potential difference between the modules $0^*_{E_R(k)}$ and $0^{*fg}_{E_R(k)}$. We refer the reader to the discussion at the beginning of \cite[Section~2]{Aberbach2002} for details.

\begin{lemma}\label{lemma tight closure vs finitistic tight closure} Let $(R,\fm,k)$ be a complete Cohen-Macaulay local ring of prime characteristic $p>0$ and of Krull dimension $d$. Let $J_1\subsetneq R$ be a choice of canonical ideal and $x_1,\ldots, x_d$ a system of parameters such that $x_1\in J$. Make the following identification of $E_R(k)$:
\[
E_R(k)\cong \varinjlim \left(\frac{R}{(x_1^{t-1}J_1,x_2^t,\cdots x_d^t)}\xrightarrow{\cdot x_1\cdots x_d}\frac{R}{(x_1^{t}J_1,x_2^{t+1},\cdots x_d^{t+1})}\right)
\]
If $\eta=[r+(x_1^{t-1}J_1,x_2^t,\cdots x_d^t)]\in E_R(k)$ then 
\begin{enumerate}
    \item $\eta\in 0^{*fg}_{E_R(k)}$ if and only if there exists a $c\in R^{\circ}$ and $s\in \NN$ such that for all $e\in\NN$
    \[
    c(r(x_1 x_2\cdots x_d)^s)^{p^e}\in (x_1^{s+t-1}J_1,x_2^{s+t},\ldots,x_d^{s+t})^{[p^e]};
    \]
    \item $\eta\in 0^{*}_{E_R(k)}$ if and only if there exists a $c\in R^{\circ}$ such that for all $e\in\NN$ there exists an $s=s(e)$ such that
    \[
    c(r(x_1x_2\cdots x_d)^s)^{p^e}\in (x_1^{s+t-1}J_1,x_2^{s+t},\ldots,x_d^{s+t})^{[p^e]}.
    \]
\end{enumerate}
\end{lemma}

\subsection{Local Cohomology Bounds} We will relate the modules $0^{*fg}_{E_R(k)}$ and $0^*_{E_{R}(k)}$ in Lemma~\ref{lemma tight closure vs finitistic tight closure} through the language of local cohomology bounds.  To this end, suppose that $M$ is a module over a ring $R$ and $\underline{x}=x_1,\ldots ,x_d$ a sequence of elements. For each $j\in \N$ let $\underline{x}^j=x_1^j,\ldots, x_d^j$ and for each pair of integers $j_1\leq j_2$ let $\tilde{\alpha}^\bullet_{M;x_i;j_1;j_2}$ be the map of Kosul cocomplexes
\[
\begin{xymatrix}
{
0 \ar[r] & M \ar[r]^{\cdot x_i^{j_1}}\ar[d]^{=} & M \ar[r]\ar[d]^{\cdot x_i^{j_2-j_1}} & 0 \\
0 \ar[r] & M \ar[r]^{\cdot x_i^{j_2}} & M \ar[r] & 0 
}
\end{xymatrix}
\]
Let $\tilde{\alpha}^\bullet_{M;\underline{x};j_1;j_2}$ be the following product:
\[
\tilde{\alpha}^\bullet_{M;\underline{x};j_1;j_2}:= \tilde{\alpha}^\bullet_{R;x_1;j_1;j_2}\otimes \tilde{\alpha}^\bullet_{R;x_2;j_1;j_2}\otimes \cdots \otimes \tilde{\alpha}^\bullet_{R;x_d;j_1;j_2}\otimes M.
\]
Then $\tilde{\alpha}^\bullet_{M;\underline{x};j_1;j_2}$ is a map of Koszul cocomplexes
\[
K^\bullet(\underline{x}^{j_1};M)\xrightarrow{\tilde{\alpha}^\bullet_{M;\underline{x};j_1;j_2}} K^\bullet(\underline{x}^{j_2};M).
\]
Let $\alpha^i_{M;\underline{x};j_1;j_2}$ denote the induced map of Koszul cohomologies
\[
H^i(\underline{x}^{j_1};M)\xrightarrow{\alpha^i_{M;\underline{x};j_1;j_2}} H^i(\underline{x}^{j_2};M).
\]
Then
\[
\varinjlim_{j_1\leq j_2} \left(H^i(\underline{x}^{j_1};M)\xrightarrow{\alpha^i_{M;\underline{x};j_1;j_2}} H^i(\underline{x}^{j_2};M)\right)\cong H^i_{(\underline{x})R}(M)
\]
by \cite[Theorem~3.5.6]{BrunsHerzog}.

Denote by $\alpha^i_{M;\underline{x};j;\infty}$ the map
\[
H^i(\underline{x}^{j};M)\xrightarrow{\alpha^i_{M;\underline{x};j;\infty}}H^i_{(\underline{x})A}(M).
\]
Observe that $\eta\in \ker (\alpha^i_{M;\underline{x};j;\infty})$ if and only if there exists some $k\geq 0$ such that $\eta\in \ker(\alpha^i_{M;\underline{x};j;j+k}).$ If $\eta\in  \ker (\alpha^i_{M;\underline{x};j;\infty})$ we let
\[
\epsilon_{\underline{x}^j}^i(\eta)=\min\{k\mid \eta \in  \ker (\alpha^i_{M;\underline{x};j;j+k})\}.
\]

\begin{definition}\label{Key definition} Let $R$ be a ring, $\underline{x}=x_1,\ldots, x_d$ a sequence of elements in $R$, and $M$ an $R$-module. The $i$th local cohomology bound of $M$ with respect to the sequence of elements $\underline{x}$ is 
\[
\lcb_i(\underline{x};M)=\sup\{\epsilon_{\underline{x}^j}^i(\eta)\mid \eta \in  \ker (\alpha^i_{M;\underline{x};j;\infty})\mbox{ for some }j \}\in \N\cup \{\infty\}.
\]
\end{definition}

Observe that if $M$ is an $R$-module and $\underline{x}$ is a sequence of elements, then $\lcb_i(\underline{x};M)=N<\infty$ simply means that if $\eta\in H^i(\underline{x}^j;M)$ represents the $0$-element in the direct limit 
\[
\varinjlim_{j_1\leq j_2} \left(H^i(\underline{x}^{j_1};M)\xrightarrow{\alpha^i_{M;\underline{x};j_1;j_2}} H^i(\underline{x}^{j_2};M)\right)\cong H^i_{(\underline{x})R}(M)
\]
then $\alpha^i_{M;\underline{x};j;j+N}(\eta)$ is the $0$-element of the Koszul cohomology group $H^{i}(\underline{x}^{j+N};M)$. Therefore finite local cohomology bounds correspond to a uniform bound of annihilation of zero elements in a choice of direct limit system defining a local cohomology module. It would be interesting to understand when local cohomology bounds are finite.

\subsection{Basic properties of local cohomology bounds} Our study of local cohomology bounds begins with two useful observations.

\begin{lemma}\label{lemma lcb power of elements} Let $R$ be a commutative Noetherian ring, $M$ an $R$-module, and $\underline{x}=x_1,\ldots, x_d$ a sequence of elements, then $\lcb_{i}(\underline{x}^j;M)\leq \lcb_i(\underline{x};M)$. Furthermore, $\lcb_i(\underline{x};M)\leq jm$ for some integers $j,m$ if and only if $\lcb_{i}(\underline{x}^j;M)\leq m$ where $\underline{x}^j$ is the sequence of elements $x_1^j,\ldots,x_d^j$.
\end{lemma}

 \begin{proof} One only has to observe that $\alpha^i_{M;\underline{x}^j;k,k+m}=\alpha^i_{M;\underline{x};jk,jk+jm}$.
 \end{proof}

If $x_1,\ldots, x_d$ is a sequence of elements in a ring $R$ and if $x_1M=0$ for some $R$-module $M$ then the short exact sequence of Koszul cocomplexes
\[
0\to K^\bullet(x_2,\ldots,x_d;M)(-1)\to K^\bullet(x_1,x_2,\ldots,x_d;M)\to K^\bullet(x_2,\ldots,x_d;M)\to 0
\]
is split and therefore $H^i(x_1,x_2,\ldots, x_d;M)\cong H^i(x_2,\ldots, x_d;M)\oplus H^{i-1}(x_2,\ldots, x_d;M)$. The content of the following lemma is a description of the behavior of the maps $\alpha^{i}_{M;x_1,x_2,\ldots,x_d; j,j+k}$ with respect to these isomorphisms of Koszul cohomologies.

\begin{lemma}\label{lemma on koszul complex containing 0 element}
Let $R$ be a commutative Noetherian ring, $M$ an $R$-module, and $x_1,x_2,\ldots,x_d$ a sequence of elements such that $x_1M=0$. If $i,j,k\in \N$ then
\[
H^i(x_1^j,x_2^j,\ldots,x_d^j; M)\cong H^i(x_2^j,\ldots,x_d^j;M)\oplus H^{i-1}(x_2^j,\ldots,x_d^j;M)
\]
and the map $\alpha_{M;x_1,x_2,\ldots, x_d;j,j+k}$ is the direct sum of $\alpha^{i}_{M;x_2,\ldots,x_d; j,j+k}$ and the $0$-map.
\end{lemma}

\begin{proof}
Let $(F^\bullet,\partial^\bullet)$ be the Koszul cocomplex $K^\bullet(x_2^j,\ldots,x_d^j;R)$ and let $(G^\bullet,\delta^\bullet)$ be the Koszul cocomplex $K^\bullet(x_1^j;R)$. Let
\[
(L^\bullet,\gamma^\bullet)=K^\bullet(x_1^j,x_2^j,\ldots,x_d^j;R)\cong K^\bullet(x_1^j;R)\otimes K^\bullet(x_2^j,\ldots,x_d^j;R).
\]
Then $L^i\cong (G^0\otimes F^i)\oplus (G^1\otimes F^{i-1})\cong F^i\oplus F^{i-1}$. We abuse notation and let $\cdot x_1^j$ denote the multiplication map on $F^i$. The map $\gamma^i$ can be thought of as
\[
\gamma^i=\begin{pmatrix}\partial^i & 0\\ \pm x_1^j & \partial^{i-1}\end{pmatrix}:F^i\oplus F^{i-1}\to F^{i+1}\oplus F^{i}.
\] 
In particular, if we apply $-\otimes_R M$ the map $\cdot \pm x_1^j\otimes M$ is the $0$-map and therefore the $i$th map of the Koszul cocomplex $K^i(x_1^j,x_2^j,\ldots,x_d^j;M)$ is the direct sum of maps $(\partial^i\otimes M)\oplus (\partial^{i-1}\otimes M)$. In particular
\[
H^i(x_1^j,x_2^j,\ldots,x_d^j; M)\cong H^i(x_2^j,\ldots,x_d^j;M)\oplus H^{i-1}(x_2^j,\ldots,x_d^j;M).
\]
To see that $\alpha_{M;x_1,x_2,\ldots, x_d;j,j+k}$ is the direct sum of $\alpha^{i}_{M;x_2,\ldots,x_d; j,j+k}$ and the $0$-map is similar to above argument but uses the fact that
\[
\tilde{\alpha}^\bullet_{M;x_1,x_2,\ldots,x_d;j;j+k}=\tilde{\alpha}^\bullet_{R;x_2,\ldots,x_d;j;j+k}\otimes \tilde{\alpha}^\bullet_{R;x_1;j;j+k}\otimes M
\]
and $ \tilde{\alpha}^1_{R;x_1;j;j+k}\otimes M=0$.
\end{proof}

A particularly useful corollary of Lemma~\ref{lemma on koszul complex containing 0 element} is the following:

 \begin{corollary}\label{corollary 0 map of Koszul cohomology} Let $R$ be a commutative Noetherian ring and $M$ an $R$-module. Suppose $x_1,\ldots,x_d$ is a sequence of elements, $1\leq i\leq d$, and $(x_1,\ldots,x_{d-i})M=0$. If $j,k\in \N$ then
 \[
 \alpha^\ell_{M;x_1,\ldots,x_d;j,j+k}:H^{\ell}(x^j_1,\ldots,x^j_d;M)\to  H^{\ell}(x^{j+k}_1,\ldots,x^{j+k}_d;M)
 \]
 is the $0$-map for all $\ell\geq i+1$. In particular, $\lcb_\ell(x_1,\ldots,x_d;M)=1$ for all $\ell \geq i+1$.
 \end{corollary}
 
 \begin{proof}
By multiple applications of Lemma~\ref{lemma on koszul complex containing 0 element} it is enough to observe that
\[
H^{\ell}(x^j_{d-i+1},\ldots, x_d^j;M)=0.
\]
This is clearly the case since $x^j_{d-i+1},\ldots, x_d^j$ is a list of $i$ elements and we are examining an $\ell \geq i+1$ Koszul cohomology of $M$ with respect to this sequence.
 \end{proof}
 
Suppose $0\to M_1\to M_2\to M_3\to 0$ is a short exact sequence of $R$-modules. The next two properties of local cohomology bounds we record allow us to compare the local cohomology bounds of the modules appearing in the short exact sequence. Proposition~\ref{second proposition 0 map Koszul cohomology} allows us to effectively compare the local cohomology bounds of two of the terms in the sequence  provided a subset of the elements in the sequence of elements defining Koszul cohomology annihilates the third. Proposition~\ref{lcb and ses} compares the the local cohomology bounds of two of the terms in the short exact whenever the sequence of elements defining Koszul cohomology is a regular sequence on the third module.

 \begin{proposition}\label{second proposition 0 map Koszul cohomology} Let $R$ be a commutative Noetherian ring and 
 \[
 0\to M_1\to M_2\to M_3\to 0
 \]
 a short exact sequence of $R$-modules. Let $\underline{x}=x_1,\ldots,x_d$ be a sequence of elements of $R$.
 \begin{enumerate}
 \item If $(x_1,\ldots, x_{d-j})M_1=0$ then for all $\ell\geq j+1$
 \[
 \lcb_\ell(\underline{x};M_2)\leq \lcb_\ell(\underline{x};M_3)+1.
 \]
 \item If $(x_1,\ldots, x_{d-j})M_2=0$ then for all $\ell\geq j+1$
 \[
 \lcb_\ell(\underline{x};M_3)\leq \lcb_{\ell+1}(\underline{x};M_1)+1.
 \]
  \item If $(x_1,\ldots, x_{d-j})M_3=0$ then for all $\ell\geq j+2$
 \[
 \lcb_\ell(\underline{x};M_1)\leq \lcb_{\ell}(\underline{x};M_2)+1.
 \]
 \end{enumerate}
 \end{proposition}

 \begin{proof} For each integer $j\in \N$ let $\underline{x}^j$ denote the sequence of elements $x_1^j,x_2^j,\ldots, x_d^j$. For (1) we consider the following commutative diagram, whose middle row is exact:
 \[
 \begin{tikzcd}
\,&  H^\ell(\underline{x}^j;M_2)\arrow{r}\arrow{d}{\alpha^\ell_{M_2;\underline{x};j;j+k}}  &  H^\ell(\underline{x}^j;M_3)\arrow{d}{\alpha^\ell_{M_3;\underline{x};j;j+k}} \\
  H^{\ell}(\underline{x}^{j+k};M_1)\arrow{r}\arrow{d}{\alpha^\ell_{M_1;\underline{x};j+k;j+k+1}}  &  H^\ell(\underline{x}^{j+k};M_2)\arrow{r}\arrow{d}{\alpha^\ell_{M_2;\underline{x};j+k;j+k+1}} &  H^\ell(\underline{x}^{j+k};M_3)\\
   H^\ell(\underline{x}^{j+k+1};M_1)\arrow{r} & H^\ell(\underline{x}^{j+k+1};M_2)
 \end{tikzcd}
 \]
By Corollary~\ref{corollary 0 map of Koszul cohomology} the map $\alpha^\ell_{M_1;\underline{x};j+k;j+k+1} $ is the $0$-map for all $\ell\geq j+1$. A straightforward diagram chase of the above diagram, which follows an element $\eta\in \ker(\alpha^\ell_{M_2;\underline{x};j;j+k'})$ for some $k'$, shows that $\eta\in \ker(\alpha^\ell_{M_2;\underline{x};j;j+k+1})$ whenever $k\geq \lcb_\ell(\underline{x};M_3)$. In particular, $\lcb_\ell(\underline{x};M_2)\leq  \lcb_\ell(\underline{x};M_3)+1$. 

Statements $(2)$ and $(3)$ follow in a similar manner and the details are left to the reader.
\end{proof}

 

\begin{proposition}\label{lcb and ses} Let $R$ be a commutative Noetherian ring, $0\to M_1\to M_2\to M_3\to 0$ a short exact sequence of $R$-modules, and $\underline{x}=x_1,\ldots,x_d$ a sequence of elements in $R$.
\begin{enumerate}
\item If $\underline{x}$ is a regular sequence on $M_1$ then $\lcb_i(\underline{x};M_2)=\lcb_{i}(\underline{x};M_3)$ for all $i\leq d-1$.
\item If $\underline{x}$ is a regular sequence on $M_2$ then $\lcb_i(\underline{x};M_3)=\lcb_{i+1}(\underline{x};M_1)$ for all $i\leq d-1$.
\item If $\underline{x}$ is a regular sequence on $M_3$ then $\lcb_i(\underline{x};M_1)=\lcb_{i}(\underline{x};M_2)$ for all $i\leq d$.
\end{enumerate}
\end{proposition}

\begin{proof} The proofs of the three statements are very similar to one another and we only provide the details of (1). 

\textbf{Proof of (1)}: For $i<d$ we have $H^i(\underline{x}^j;M_1)=0$ and therefore if $i\leq d-2$ there are commutative diagrams
\[
\begin{tikzcd}
H^i(\underline{x}^j;M_2)\arrow{r}{\cong}\arrow{d}{\alpha^i_{M_2;\underline{x};j;j+k}} & H^i(\underline{x}^j;M_3)\arrow{d}{\alpha^i_{M_3;\underline{x};j;j+k}} \\
H^i(\underline{x}^{j+k};M_2)\arrow{r}{\cong} & H^{i}(\underline{x}^{j+k};M_3)
\end{tikzcd}
\]
whose horizontal arrows are isomorphisms. It readily follows that $\lcb_i(\underline{x};M_2)=\lcb_{i}(\underline{x};M_3)$ whenever $i\leq d -2$. Because $\underline{x}$ is a regular sequence on $M_1$ we have that the maps $\alpha^d_{M_1,\underline{x},j,j+k}$ are injective. Conside the following commutative diagrams whose rows are exact:
\[
\begin{tikzcd}
0 \arrow{r} & H^{d-1}(\underline{x}^j; M_2) \arrow{r}{\pi_j} \arrow{d}{\alpha^{d-1}_{M_2;\underline{x};j;j+k}} & H^{d-1}(\underline{x}^j; M_3)\arrow{r}{\delta_j} \arrow{d}{\alpha^{d-1}_{M_3;\underline{x};j;j+k}} & H^{d}(\underline{x}^j; M_1) \arrow{d}{\alpha^{d}_{M_1;\underline{x};j;j+k}} \\ 
0 \arrow{r} & H^{d-1}(\underline{x}^{j+k}; M_2)\arrow{r}{\pi_{j+k}} & H^{d-1}(\underline{x}^{j+k}; M_3) \arrow{r}{\delta_{j+k}} & H^{d}(\underline{x}^{j+k}; M_1)
\end{tikzcd}
\]
If  $\eta\in \ker(\alpha^{d-1}_{M_2;\underline{x};j,j+k})$ then $\pi_j(\eta)\in \ker(\alpha^{d-1}_{M_3;\underline{x};j,j+k})$. The maps $\pi_{j+k}$ are injective. Therefore $\alpha^{d-1}_{M_2;\underline{x};j,j+k}(\eta)=0$ whenever $k\geq \lcb_{d-1}(\underline{x};M_3)$ and hence $\lcb_{d-1}(\underline{x};M_2)\leq \lcb_{d-1}(\underline{x};M_3)$.

To show that $\lcb_{d-1}(\underline{x};M_2)\geq \lcb_{d-1}(\underline{x};M_3)$ consider an element $\eta\in \ker(\alpha^{d-1}_{M_3;\underline{x};j;j+k}).$ Then $\delta_j(\eta)\in  \ker(\alpha^{d}_{M_1;\underline{x};j;j+k}) $. But the maps $\alpha^{d}_{M_1;\underline{x};j;j+k}$ are injective and therefore $\delta_j(\eta)=0$. In particular, $\eta=\pi_j(\eta')$ for some $\eta'\in H^{d-1}(\underline{x}^j;M_2)$. The maps $\pi_{j+k}$ are all injective. Therefore $\eta'\in \ker(\alpha^{d-1}_{M_1;\underline{x};j;j+k})$ and it follows that $\alpha^{d-1}_{M_2;\underline{x};j;j+k}(\eta)=0$ whenever $k\geq \lcb_{d-1}(\underline{x};M_2)$. Therefore $\lcb_{d-1}(\underline{x};M_2)\geq \lcb_{d-1}(\underline{x};M_3)$ and hence $\lcb_{d-1}(\underline{x};M_2)= \lcb_{d-1}(\underline{x};M_3)$. This completes the proof of $(1)$.
\end{proof}

 \section{Equality of test ideals}\label{section weak implies strong method}

Theorem~\ref{Main Theorem 1} is a consequence of Theorem~\ref{theorem how to make test ideals agree using lcbs} and Theorem~\ref{theorem on how to make test ideals agree 1}. Theorem~\ref{theorem how to make test ideals agree using lcbs} is an explicit relationship between local cohomology bounds and equality of test ideals. Theorem~\ref{theorem on how to make test ideals agree 1}, when paired with Proposition~\ref{proposition technical lcb}, provides the needed local cohomology bounds described in Theorem~\ref{theorem on how to make test ideals agree 1} whenever we are able to linearly compare the annihilators of $H^i_\fm(R/I^{(n)})$ as $n\to \infty$ and $I\cong \omega_R^{(-m)}$ is a multiple of an anti-canonical ideal.

\subsection{Local Cohomology bounds and equality of test ideals}

The content of the following lemma can be pieced together by work of the first author in  \cite{Aberbach2002}. We refer the reader to \cite[Lemma~6.7]{PolstraTucker} for a direct presentation of the lemma.\footnote{In \cite[Lemma~6.7]{PolstraTucker} there is an assumption that $R$ is complete. The lemma claims equality among certain colon ideals, and equality of ideals can be checked after completion as $R\to \widehat{R}$ is faithfully flat.}

\begin{lemma}
\label{colonslemma}
Suppose that $(R,\fm,k)$ is a Cohen-Macaulay local normal domain of dimension $d$, and $J\subseteq R$ an ideal of pure height $1$.  Let $x_1, \ldots, x_d \in R$ be a system of parameters for $R$, assume that $x_1\in J$, and fix $e \in \N$.
\begin{enumerate}
\item
If $x_2 J \subseteq a_2 R$ for some $a_2 \in J$, then for any non-negative integers $N_2, \ldots, N_d$ with $N_2 \geq 2$, we have that
\[
\begin{array}{ll}
 & ((J^{(p^{e})},x_2^{N_2 p^e}, x_3^{N_3 p^e}, \ldots, x_d^{N_d p^e}):x_2^{(N_2 - 1)p^e} ) \\ = &
 ((J^{[p^{e}]},x_2^{N_2 p^e}, x_3^{N_3 p^e}, \ldots, x_d^{N_d p^e}):x_2^{(N_2 - 1)p^e} ) \\ =  & ((J^{[p^{e}]},x_2^{2 p^e}, x_3^{N_3 p^e}, \ldots, x_d^{N_d p^e}):x_2^{p^e} ).
\end{array}
\]
\item
Suppose $x_3^mJ^{(m)} \subseteq a_3 R\subseteq J^{(m)}$ for some $m\in \NN$, then for any non-negative integers $N_2, \ldots, N_d$ with $N_3 \geq 2$, we have that
\begin{equation*}
\begin{array}{ll}
& ((J^{(p^{e})},x_2^{N_2 p^e}, x_3^{N_3p^e}\ldots,  x_d^{N_d p^e}):  x_3^{(N_3 -1) p^e}) \\  \subseteq 
& ((J^{(p^{e})},x_2^{N_2 p^e},x_3^{2p^e},\ldots,  x_d^{N_d p^e}): x_1^{m} x_3^{p^e}).
\end{array}
\end{equation*}

\end{enumerate}
\end{lemma}

\begin{theorem}
\label{theorem how to make test ideals agree using lcbs}
Let $(R,\fm,k)$ be a local normal Cohen-Macaulay domain of Krull dimension $d$ and of prime characteristic $p>0$. Assume that $R$ has a test element.  Let $J_1\subseteq R$ be a choice of canonical ideal. Suppose $x_1,\ldots,x_d$ is a system of parameters of $R$, $x_1\in J_1$, and suppose that the following conditions are met:
\begin{itemize}
\item There exists element $a_2\in J_1$ and $a_3\in J_1^{(m)}$ such that $x_2J_1\subseteq a_2 R$ and $x_3^mJ_1^{(m)}\subseteq a_3R$;\footnote{This property is automatic if $R_P$ is $\mathbb{Q}$-Gorenstein for each height $2$ prime ideal $P\in \Spec(R)$. Recall that a local normal Cohen-Macaulay domain $R$  with canonical ideal $J\subseteq R$ is $\mathbb{Q}$-Gorenstein if there exists a $m\geq 1$ such that $J^{(m)}$ is a principal ideal. If $W_1$ is the complement of the union of the minimal components of $J_1$, then $J_1R_{W_1}$ is principally generated by an element $a_2\in J_1$, hence $x_2$ can be chosen with the property $x_2J_1\subseteq (a_2)\subseteq J_1$. If $W_2$ is then the complement of the union of the minimal components of $(J_1,x_2)$ then $R_{W_2}$ is a semi-local domain. Hence we can choose an $m$ so that $J^{(m)}R_{W_2}$ is principally generated by an element $a_3\in J_1^{(m)}$ by \cite[Theorem~60]{Kaplansky}. There then exists parameter element $x_3$ so that $x_3J^{(m)}\subseteq (a_3)\subseteq J^{(m)}$. We opt to use the containment $x^m_3J^{(m)}\subseteq (a_3)\subseteq J^{(m)}$ to ease computational complexity of the proof.}
\item For each $e\in \N$ there exists an integer $\ell$ such that
\[
\lcb_{d-1}(x^{\ell}_2,x^{\ell }_3,x_4,\ldots,x_d; R/J_1^{(mp^e+1)})\leq p^e+1.
\]
\end{itemize}
Then $0^{*fg}_{E_R(k)}=0^*_{E_R(k)}$. 
\end{theorem}

\begin{proof}
Identify $E_R(k)$ as 
\[
E_R(k)\cong \varinjlim\left(\frac{R}{(x_1^{t-1}J_1,x_2^t,\ldots,x_d^t)}\xrightarrow{\cdot x_1x_2\cdots x_d}\frac{R}{(x_1^tJ_1,x_2^{t+1},\ldots,x_d^{t+1})}\right).
\]
Suppose that $\eta=[r+(x_1^{t-1}J_1,x_2^t,\ldots,x_d^t)]\in 0^*_{E_R(k)}$.  Equivalently, there exists a $c\in R^\circ$ such that for all $e\in\NN$
\[
0=c\eta^{p^e}=[cr^{p^e}+(x_1^{t-1}J_1,x_2^t,\ldots,x_d^t)^{[p^e]}]\in F^e(E_R(k))\cong \varinjlim\left(\frac{R}{(x_1^{t-1}J_1,x_2^t,\ldots,x_d^t)^{[p^e]}}\right).
\]
Let $J=x_1^{t-1}J_1$ and consider the local cohomology module
\[
H^{d-1}_\fm\left(\frac{R}{J^{[p^e]}}\right)=\varinjlim\left(\frac{R}{J^{[p^e]}+(x_2^t,\ldots,x_d^t)}\xrightarrow{\cdot x_2\cdots x_d}\frac{R}{J^{[p^e]}+(x_2^{t+1},\ldots,x_d^{t+1})}\right).
\]
\begin{claim}
\label{0 element claim}
Let
\[
\alpha^{p^e}=[r^{p^e}+(x_2^{tp^e},\ldots,x_d^{tp^e})]\in H^{d-1}_\fm\left(\frac{R}{J^{[p^e]}}\right),
\]
then $c\alpha^{p^e}=[cr^{p^e}+(x_2^t,\ldots,x_d^t)^{[p^e]}]$ is the $0$-element of $H^{d-1}_\fm(R/J^{[p^e]})$.
\end{claim}
\begin{proof}[Proof of Claim]
The element $[cr^{p^e}+(x_1^{t-1}J_1,x_2^t,\ldots,x_d^t)^{[p^e]}]$ is the $0$-element of 
\[
\varinjlim\left(\frac{R}{(x_1^{t-1}J_1,x_2^t,\ldots,x_d^t)^{[p^e]}}\right).
\]
Therefore there exists an $s\in\NN$ such that 
\[
cr^{p^e}(x_1x_2\cdots x_d)^{sp^e}\in (x_1^{t+s-1}J_1,x_2^{t+s},\ldots,x_d^{t+s})^{[p^e]}=(x_1^{(t+s-1)p^e}J_1^{[p^e]},x_2^{(t+s)p^e},\ldots,x_d^{(t+s)p^e}).
\]
So there exists an element $j_1\in J_1^{[p^e]}$ such that 
\[
cr^{p^e}(x_1x_2\cdots x_d)^{sp^e}-x_1^{(t+s-1)p^e}j_1\in (x_2^{(t+2)p^e},\ldots,x_d^{(t+s)p^e}).
\]
The sequence $x_1,x_2,\ldots,x_d$ is a regular sequence and so
\[
cr^{p^e}(x_2\cdots x_d)^{sp^e}-x_1^{(t-1)p^e}j_1\in (x_2^{(t+s)p^e},\ldots,x_d^{(t+s)p^e}).
\]
Hence
\[
cr^{p^e}(x_2\cdots x_d)^{sp^e}\in (x_1^{(t-1)p^e}J_1^{[p^e]},x_2^{(t+s)p^e},\ldots,x_d^{(t+s)p^e})=(J^{[p^e]},x_2^{(t+s)p^e},\ldots,x_d^{(t+s)p^e}),
\]
which proves the claim.
\end{proof}

Choose $e_0\in\NN_{\geq 1}$ so that $p^e\geq mp^{e-e_0}+1$ for all $e\gg 0$. If $e\gg 0$ then
\[
J^{[p^e]}\subseteq J^{(p^e)}\subseteq J^{(mp^{e-e_0}+1)}.
\]
Fix $e\gg 0$ and consider the local cohomology module
\[
H^{d-1}_\fm\left(\frac{R}{J^{(mp^{e-e_0}+1)}}\right)\cong \varinjlim\left(\frac{R}{(J^{(mp^{e-e_0}+1)},x_2^{t},\ldots,x_d^{t})}\right).
\]
Let $\tilde{\alpha}^{p^e}$ denote the image of $\alpha^{p^e}$ in $H^{d-1}_\fm(R/J^{(mp^{e-e_0}+1)})$ induced by the projection $R/J^{[p^e]}\to R/J^{(mp^{e-e_0}+1)}$. By Claim~\ref{0 element claim}
\[
0=c\tilde{\alpha}^{p^e}=[cr^{p^e}+(x_2^{tp^e},\ldots,x_d^{tp^e})]\in H^{d-1}_\fm\left(\frac{R}{J^{(mp^{e-e_0}+1)}}\right)\cong  \varinjlim\left(\frac{R}{J^{(mp^{e-e_0}+1)}+(x_2^{t},\ldots,x_d^{t})}\right).
\]
There are short exact sequences
\[
0\to \frac{R}{J_1^{(mp^{e-e_0}+1)}}\xrightarrow{\cdot x_1^{(t-1)(mp^{e-e_0}+1)}}\frac{R}{J^{(mp^{e-e_0}+1)}}\to \frac{R}{x_1^{(t-1)(mp^{e-e_0}+1)}R}\to 0.
\]
Let $\ell$ be a choice of integer, which depends on $e-e_0$, as in the statement of the theorem. The sequence $x_2^\ell,x_3^\ell,x_4, \ldots,x_d$ is a regular sequence on $R/x_1^{(t-1)(mp^{e-e_0}+1)}R$. By (3) of Proposition~\ref{lcb and ses} we have that 
\[
\lcb_{d-1}(x^{\ell}_2,x^{\ell }_3,x_4,\ldots,x_d; R/J_1^{(mp^{e-e_0}+1)})=\lcb_{d-1}(x^{\ell}_2,x^{\ell }_3,x_4,\ldots,x_d; R/J^{(mp^{e-e_0}+1)}),
\]
and so by assumption
\[
\lcb_{d-1}(x^{\ell}_2,x^{\ell }_3,x_4,\ldots,x_d; R/J^{(mp^{e-e_0}+1)})\leq p^{e-e_0}+1\leq p^e.
\]

Recall that
\[
0=[cr^{p^e}+(x_2^{tp^e},\ldots,x_d^{tp^e})]=[cr^{p^e}(x^t_2x^t_3)^{(\ell-1)p^e}+(x_2^{t\ell p^e},x_3^{t\ell p^e}, x_4^{tp^e}\ldots,x_d^{tp^e})]
\]
as an element of $H^{d-1}_\fm(R/J^{(mp^{e-e_0}+1)})$.
By Lemma~\ref{lemma lcb power of elements}, 
\[
\lcb_{d-1}(x^{t\ell}_2,x^{t\ell }_3,x^t_4,\ldots,x^t_d; R/J^{(mp^{e-e_0}+1)})\leq \lcb_{d-1}(x^{\ell}_2,x^{\ell }_3,x_4,\ldots,x_d; R/J^{(mp^{e-e_0}+1)})\leq p^e,
\]
so
\begin{equation}\label{implied containment from lcb}
c(rx^t_4\cdots x^t_d)^{p^e}(x^t_2x^t_3)^{\ell p^e}\in (J^{(mp^{e-e_0}+1)},x_2^{t(\ell+1)p^e},x_3^{t(\ell+1)p^e}, x_4^{2tp^e},\ldots, x_d^{2tp^e}).
\end{equation}

Notice that $x_1^{tp^{e_0}}\in J^{(p^{e_0})}$  and so 
\[
x_1^{tp^{e_0}}J^{(mp^{e-e_0}+1)}\subseteq x_1^{tp^{e_0}}J^{(p^{e-e_0})}\subseteq  J^{(p^e)}.
\]
We therefore multiply the containment (\ref{implied containment from lcb}) by $x_1^{tp^{e_0}}$ and obtain that
\[
cx_1^{tp^{e_0}}(rx^t_4\cdots x^t_d)^{p^e}(x^t_2x^t_3)^{\ell p^e} \in (J^{(p^e)},x_2^{t(\ell+1)p^e},x_3^{t(\ell+1)p^e}, x_4^{2tp^e},\ldots, x_d^{2tp^e}).
\]
Therefore
\[
cx_1^{tp^{e_0}}(rx^t_4\cdots x^t_d)^{p^e}(x_2^t)^{\ell p^e} \in (J^{(p^e)},(x^t_2)^{(\ell+1)p^e},(x^t_3)^{(\ell+1)p^e}, x_4^{2tp^e},\ldots, x_d^{2tp^e}):x_3^{t\ell p^e}.
\]

We utilize the assumption that $x_3^mJ_1^{(m)}\subseteq a_2R\subseteq J_1^{(m)}$ to conclude that;
\[
x_3^mJ^{(m)}=x_3^m(x_1^{t-1}J_1)^{(m)}=x_1^{(t-1)m}x_3^mJ_1^{(m)}\subseteq x_1^{(t-1)m}a_2R\subseteq x_1^{(t-1)m}J_1^{(m)}\subseteq J^{(m)}.
\]
Therefore $x_3^{tm}J^{(m)}\subseteq x_1^{(t-1)m}a_2R\subseteq J^{(m)}$ and we apply (2) of Lemma~\ref{colonslemma} with respect to $x_3^t$ and $N_3=\ell+1$ to conclude that
\[
cx_1^{tp^{e_0}}(rx^t_4\cdots x^t_d)^{p^e}(x^t_2)^{\ell p^e} \in (J^{(p^e)},(x_2^t)^{(\ell+1)p^e},x_3^{2tp^e}, x_4^{2tp^e},\ldots, x_d^{2tp^e}):x_1^{tm}x_3^{tp^e}.
\]
Equivalently, 
\[
cx_1^{t(m+p^{e_0})}(rx_3^tx^t_4\cdots x^t_d)^{p^e}(x^t_2)^{\ell p^e} \in (J^{(p^e)},(x_2^t)^{(\ell+1)p^e},x_3^{2tp^e}, x_4^{2tp^e},\ldots, x_d^{2tp^e}).
\]
Similarly, we are able to apply (1) of Lemma~\ref{colonslemma} with respect to the element $x_2^t$ and obtain that
\[
cx_1^{t(m+p^{e_0})}(rx_2^tx_3^tx^t_4\cdots x^t_d)^{p^e} \in (J^{[p^e]},x_2^{2tp^e},x_3^{2tp^e}, x_4^{2tp^e},\ldots, x_d^{2tp^e})=(J,x_2^{2t},x_3^{2t},\ldots, x_d^{2t})^{[p^e]}.
\]
The element $cx_1^{t(m+p^{e_0})}$ does not depend on $e$ and therefore 
\[
rx_2^tx_3^tx^t_4\cdots x^t_d\in (J,x_2^{2t},x_3^{2t},\ldots, x_d^{2t})^*.
\]
In particular,
\[
\eta=[r+(x_1^{t-1}J_1,x_2^t,\ldots,x_d^t)]=[rx_2^{t}x_3^tx^t_4\cdots x^t_d+(J,x_2^{2t},\ldots,x_d^{2t})]\in 0^{*fg}_{E_R(k)}
\]
as claimed.
\end{proof}

\subsection{$S_2$-ification, higher Ext-modules, and local cohomology}

We begin with two lemmas that experts may already be aware of.

\begin{lemma}\label{lemma s2-module} Let $(S,\fm,k)$ be a Cohen-Macaulay local domain and $M$ a finitely generated $S$-module such that $\Ht(\Ann_S(M))=h$. Then $\Ext^{h}_S(M,S)$ is an $(S_2)$-module over its support.
\end{lemma}

\begin{proof}
Let $(F_\bullet,\partial_\bullet)$ be the minimal free resolution of $M$, let $(-)^*$ denote $\Hom_S(-,S)$, and consider the dual complex $(F_\bullet^*,\partial_\bullet^*)$. Because  $\Ht(\Ann_S(M))=h$ we have that the following complex is exact:
 \[
 0\to F_0^*\xrightarrow{\partial_1^*}F_1^*\to \ldots\to F_{h-1}^*\xrightarrow{\partial_{h}^*} F_h^*\to \coker(\partial_{h}^*)\to 0. 
 \]
In particular, $\depth(\coker(\partial_{h}^*))=d-h$. Moreover, there is a short exact sequence
\[
0\to \Ext^{h}_S(M,S)\to \coker(\partial_h^*)\to \im(\partial_{h+1}^*)\to 0.
\]
The module $\im(\partial_{h+1}^*)$ is torsion-free and therefore has depth at least $1$. If $d-h\geq 2$ then $\Ext^{h}_S(M,S)$ has depth at least $2$. If $d-h=1$ then the depth of $\Ext^{h}_S(M,S)$ is $1$. If $d-h=0$ then $M$ is $0$-dimensional. Therefore if $\Ht(\Ann_S(M))=h$ then $\Ext^{h}_S(M,S)$ is an $(S_2)$-module over its support.
\end{proof}

Continue to consider the ring $S$, the module $M$, and the resolution $(F_\bullet,\partial_\bullet)$ as above. Suppose further $S$ is a regular local ring and hence every finitely generated $S$-module has a finite free resolution. Consider the minimal free resolution $(G_\bullet,\delta_\bullet)$ of $\Ext^{h}_S(M,S)$. If $\depth(M)=d-h$ is maximal, then $\Ext^{h}_S(M,S)=\coker(\partial_h^*)$ and therefore $(G_\bullet,\delta_\bullet)$ is the complex
\[
 0\to F_0^*\xrightarrow{\partial_1^*}F_1^*\to \ldots\to F_{h-1}^*\xrightarrow{\partial_{h}^*} F_h^*\to 0.
\]
In particular, if $\depth(M)=d-h$ then $\Ext^{h}_S(\Ext^h_S(M,S),S)\cong M$. Suppose $\depth(M)<d-h$ and let $(F_\bullet^*,\partial_\bullet^*)_{tr}$ be the complex obtained by truncating $(F_\bullet^*,\partial_\bullet^*)$ at the $h$th spot. That is $(F_\bullet^*,\partial_\bullet^*)_{tr}$ is the minimal free resolution of $\coker(\partial_h^*)$. Then the natural inclusion $\Ext^h_S(M,S)\subseteq \coker(\partial_h^*)$ lifts to a map of complexes $(G_\bullet,\delta_\bullet)\to (F_\bullet^*,\partial_\bullet^*)_{tr}$ and therefore there is an induced natural map $M\to \Ext^h_S(\Ext^h_S(M,S),S)$. The map $M\to \Ext^h_S(\Ext^h_S(M,S),S)$ is an isomorphism whenever $M$ is a (maximal) Cohen-Macaulay module over its support.

\begin{lemma}\label{Inclusion lemma} Let $(R,\fm,k)$ be a complete local normal domain of Krull dimension $d\geq 1$ and $J\subseteq R$ a pure height $1$ ideal. Suppose $(S,\fn,k)$ is a regular local ring mapping onto $R$, $R\cong S/P$, and $\Ht(P)=h$. Then for every integer $i$ the kernel of the natural map $R/J^i\to \Ext^{h+1}_S(\Ext^{h+1}_S(R/J^i,S),S)$ is $J^{(i)}/J^i$. In particular, for every integer $i$ there is a natural inclusion $R/J^{(i)}\subseteq \Ext^{h+1}_S(\Ext^{h+1}_S(R/J^i,S),S)$. Moreover, the natural inclusion $R/J^{(i)}\subseteq \Ext^{h+1}_S(\Ext^{h+1}_S(R/J^i,S),S)$ is an isomorphism whenever localized at prime ideal $\fp\in V(J)$ such that $(R/J^{(i)})_\fp$ is Cohen-Macaulay.
 \end{lemma}
\begin{proof}
 Let $L_i\subseteq R$ be the ideal of $R$, containing $J^i$, so that $L_i/J^i$ is the kernel of $R/J^i\to \Ext^{h+1}_S(\Ext^{h+1}_S(R/J^i,S),S)$. Then $R/L_i\subseteq \Ext^{h+1}_S(\Ext^{h+1}_S(R/J^i,S),S)$. If $P$ is a prime component of $J$ then $R_P/J^iR_P$ is $0$-dimensional and therefore Cohen-Macaulay. By the above the discussion, the map under analysis is an isomorphism in the Cohen-Macaulay locus, and therefore
\[
R_P/J^iR_P=R_P/J^{(i)}R_P=R_P/{L_i}_P=\Ext_{S_P}^{h+1}(\Ext_{S_P}^{h+1}(R_P/J^{(i)}R_P,S_P),S_P)
\]
at prime $P$ which are minimal components of $J$.

If $P$ is a prime of $R$ of height $1$ which is not a component of $J$, then $R_P/J^iR_P=0$ and the identifications above remain true. Therefore the height $1$ components of the ideal of $L_i$ are precisely the height $1$ components of $J^i$, i.e. $J^{(i)}$. To conclude that $L_i=J^{(i)}$ it remains to show that the ideal $L_i$ does not have embedded components. The module $\Ext^{h+1}_S(\Ext^{h+1}_S(R/J^i,S),S)$ is an $(S_2)$ $R/J^{i}$-module and $R/L_i$ is a submodule. Therefore $R/L_i$ is an $(S_1)$-module and hence $L_i$ cannot have an embedded component.

We have proven the first claim of the lemma that for each $i\in \mathbb{N}$ there is a natural inclusion $R/J^{(i)}\subseteq \Ext^{h+1}_S(\Ext^{h+1}_S(R/J^i,S),S)$. It remains to check that this inclusion is an isomorphism whenever localized at prime ideal $\fp\in V(J)$ such that $(R/J^{(i)})_\fp$ is Cohen-Macaulay. Indeed, $R/J^i\to R/J^{(i)}$ induces an isomorphism 
\[
\Ext^{h+1}_S(R/J^{(i)},S)\xrightarrow{\cong} \Ext^{h+1}_S(R/J^{i},S)
\]
as $J^{(i)}/J^i$ is not supported at any height $h+1$ component of $S$. Therefore the inclusion $R/J^{(i)}\subseteq \Ext^{h+1}_S(\Ext^{h+1}_S(R/J^i,S),S)$ is the same as 
\[
R/J^{(i)}\to \Ext^{h+1}_S(\Ext^{h+1}_S(R/J^{(i)},S),S)
\]
and this map is an isomorphism in the Cohen-Macaulay locus by the discussion preceding the lemma.
\end{proof}
 
 We record a corollary of Lemma~\ref{Inclusion lemma} for future reference.
 
 \begin{corollary}\label{Inclusion lemma corollary}
 Let $(R,\fm,k)$ be a complete local normal domain, $\Q$-Gorenstein in codimension $2$, and $J_1\subsetneq R$ a choice of canonical ideal. Let $m\in\N$ be an integer such that $J_1^{(m)}$ is principal in codimension $2$. Suppose $(S,\fn,k)$ is a regular local ring mapping onto $R$, $R\cong S/P$, and $\Ht(P)=h$. Then for every integer $i$ the natural inclusion $R/J_1^{(mi+1)}\to \Ext^{h+1}_S(\Ext^{h+1}_S(R/J_1^{mi+1},S),S)$ is an isomorphism whenever localized at a prime ideal of $R$ of height $2$ or less.
 \end{corollary}
 \begin{proof}
If $\fp$ is prime ideal of height $2$ or less then $R_\fp$ is Cohen-Macaulay and hence $R_\fp/J_1R_\fp$ is Gorenstein of dimension at most $1$. The Corollary follows by Lemma~\ref{Inclusion lemma} as $J_1^{(mi+1)}R_\fp\cong J_1R_\fp$ is a canonical ideal whenever $\fp$ is a prime of $R$ of height $2$ or less.
 \end{proof}

The next proposition and theorem provide the linear bound of top local cohomology bounds of the family of $R$-modules $\left\{R/J_1^{(mp^e+1)}\right\}$ described in Theorem~\ref{theorem how to make test ideals agree using lcbs} whenever there exists an ideal $I\subseteq R$ of pure height $1$ and parameters $x_1,\ldots,x_d$ with the following properties:
\begin{enumerate}
    \item $I\cong \omega_R^{(-h)}$ for some $h\geq 1$ and $I$ is principal in codimension $2$;
    \item For each $1\leq j\leq d-2$, the ideal $\fa_j:=(x_1,\ldots,x_{d-j+1})$ is such that
    \[
    \fa_j^{[p^e]}H^j_\fm(R/I^{(p^e)})=0
    \] 
    for each $e\in\NN$.
\end{enumerate}
We first provide a lemma. In the following lemma we let $(-)^\vee$ denote the Matlis dual functor.

\begin{lemma}\label{lemma matlis duality and higher ext modules} Let $(R,\fm,k)$ be a local normal Cohen-Macaulay domain of Krull dimension $d$ and $\Q$-Gorenstein in codimension $2$. Assume that $R$ has a test element.  Let $J_1\subseteq R$ be a choice of canonical ideal and $m\in \N$ such that $J_1^{(m)}$ is principal in codimension $2$. Suppose $S$ is a regular local ring of Krull dimension $d+h$ mapping onto $R$, $R\cong S/P$, and $\Ht(P)=h$. Suppose that $I_1\subseteq R$ is an ideal of pure height $1$ with components disjoint from those of $J_1$, $I_1\cap J_1=x_1R$ is principal. Then 
\[
H^{j-1}_\fm\left(\frac{R}{I_1^{(mi)}}\right)\cong \left(\Ext^{d+h-j}_S\left(\Ext^{h+1}_S\left(\frac{R}{J_1^{(mi+1)}},S\right),S\right)\right)^\vee
\]
for all $j\leq d-2$.
\end{lemma}

\begin{proof}
If $j\leq 0$ then $H^{j-1}_\fm\left(\frac{R}{I_1^{(mi)}}\right) = \Ext^{d+h-j}_S\left(\Ext^{h+1}_S\left(\frac{R}{J_1^{(mi+1)}},S\right),S\right)=0$. In particular, we may assume that $d\geq 3$. There are isomorphisms
\begin{align}
\label{Ext-iso 1}
\Ext^{h+1}_S\left(\frac{R}{J_1^{(mi+1)}},S\right)\cong \omega_{R/J_1^{(mi+1)}}\cong \Ext_R^1\left(\frac{R}{J_1^{(mi+1)}},J_1\right).
\end{align}
Consider the short exact sequences
\[
0\to J_1^{(mi+1)}\to R\to \frac{R}{J_1^{(mi+1)}}\to 0.
\]
The ring $R$ is Cohen-Macaulay. Therefore $\Ext_R^1\left(R,J_1\right)=0$ and there is a resulting short exact sequence
\begin{align}
\label{Ext-iso 2}
0\to J_1\to \Hom_R(J_1^{(mi+1)},J_1)\to \Ext_R^1\left(\frac{R}{J_1^{(mi+1)}},J_1\right)\to 0.
\end{align}
But $I_1\cap J_1$ is principal, $I_1$ and $J_1$ have disjoint components, therefore $J_1\cong I_1^{(mi)}\cap J_1^{(mi+1)}$ and so
\begin{align}
\label{Ext-iso 3}
\Hom_R(J_1^{(mi+1)},J_1)\cong \Hom_R(J_1^{(mi+1)},I_1^{(mi)}\cap J_1^{(mi+1)})\cong I_1^{(mi)}.
\end{align}
The ideal $J_1$ is a maximal Cohen-Macaulay $R$-module and so $\Ext^{\geq h+1}_S(J_1,S)=0$. Therefore by (\ref{Ext-iso 1}), (\ref{Ext-iso 2}), and (\ref{Ext-iso 3}), if $j\leq d-2$ then 
\[
\Ext^{d+h-j}_S\left(\Ext^{h+1}_S\left(\frac{R}{J_1^{(mi+1)}},S\right),S\right)\cong \Ext^{d+h-j}_S\left(I_1^{(mi)},S\right).
\]
Consider the short exact sequence
\[
0\to I_1^{(mi)}\to R\to\frac{R}{I_1^{(mi)}}\to 0.
\]
Then
\[
\Ext^{d+h-j}_S\left(I_1^{(mi)},S\right)\cong \Ext^{d+h-(j-1)}_S\left(\frac{R}{I_1^{(mi)}},S\right).
\]
An application of Matlis duality now completes the proof as
\[
\left(\Ext^{d+h-(j-1)}_S\left(\frac{R}{I_1^{(mi)}},S\right)\right)^\vee\cong H^{j-1}_\fm\left(\frac{R}{I_1^{(mi)}}\right).
\]
\end{proof}

\begin{proposition}\label{proposition technical lcb} Let $(R,\fm,k)$ be a local normal Cohen-Macaulay domain of Krull dimension $d$ and $\Q$-Gorenstein in codimension $2$. Let $p>0$ be a natural number. Let $J_1\subseteq R$ be a choice of canonical ideal and $m\in \N$ such that $J_1^{(m)}$ is principal in codimension $2$. Suppose $S$ is a regular local ring mapping onto $R$, $R\cong S/P$, and $\Ht(P)=h$. Suppose that $I_1\subseteq R$ is an ideal of pure height $1$ with components disjoint from those of $J_1$, $I_1\cap J_1=x_1R$ is principal, and parameters $x_1,x_2\ldots,x_d$ with the property that for each $1\leq j\leq d-2$, the parameter ideal $(x_2,\ldots,x_{d-j+1})$ is such that
    \[
    (x_2^{p^e},\ldots,x^{p^e}_{d-j+1})H^j_\fm(R/I_1^{(mp^e)})=0
    \] 
    for each $e\in\NN$. Then for all $e\in \NN$
\[
\lcb_{d-1}(x^{d-3}_2,x^{d-3}_3,\ldots,x^{d-3}_d;\Ext^{h+1}_S(\Ext^{h+1}_S(R/J_1^{(mp^e+1)},S),S))\leq p^e.
\]
\end{proposition}

\begin{proof}
Let $(F_\bullet,\partial_\bullet)$ be the minimal free $S$-resolution of $\Ext^{h+1}_S(R/J_1^{(mp^e+1)},S)$. Denote by $(-)^*$ the functor $\Hom_S(-,S)$ and consider the dualized complex $(F_\bullet^*,\partial_\bullet^*)$. For every $j\geq 1$ there are short exact sequences
\[
0\to \Ext^{h+j}_S(\Ext^{h+1}_S(R/J_1^{(mp^e+1)},S),S)\to \coker(\partial_{h+j}^*)\to \im(\partial_{h+j+1}^*)\to 0
\]
and
\[
0\to \im(\partial_{h+j+1}^*)\to F_{h+j+1}^*\to \coker(\partial_{h+j+1}^*)\to 0.
\]
Let $\mathcal{J}_e$ denote the preimage of $J_1^{(mp^e+1)}$ in $S$, an ideal of height $h+1$. The $S$-module $\coker(\partial_{h+1}^*)$ has projective dimension $h+1$ and the ideal $\mathcal{J}_e$ annihilates the submodule $\Ext^{h+1}_S(\Ext^{h+1}_S(R/J_1^{(mp^e+1)},S),S)$. By prime avoidance, and abuse of notation, we may lift $\underline{x}=x_2,\ldots,x_d$ to elements of $S$ and assume that $\underline{x}$ is a regular sequence on $\coker(\partial_{h+1}^*)$ and the free $S$-modules $F_i^*$.

The module $\Ext^{h+1}_S(R/J_1^{(mp^e+1)},S)$ is an $(S_2)$-module over its support, see Lemma~\ref{lemma s2-module}. In particular, 
\[
\Ext^{h+d}_S(\Ext^{h+1}_S(R/J_1^{(mp^e+1)},S),S)=\Ext^{h+d-1}_S(\Ext^{h+1}_S(R/J_1^{(mp^e+1)},S),S)=0
\]
and 
\[
\coker(\partial^*_{h+d-2})\cong \Ext^{h+d-2}_S(\Ext^{h+1}_S(R/J_1^{(mp^e+1)},S),S).
\]
Consider the short exact sequence
\[
0\to \im(\partial_{h+d-2}^*)\to F_{h+d-2}^*\to   \Ext^{h+d-2}_S(\Ext^{h+1}_S(R/J_1^{(mp^e+1)},S),S)\to 0.
\]
By our assumptions and by Lemma~\ref{lemma matlis duality and higher ext modules}, 
\[
(x_2^{p^e},x_3^{p^e},\ldots,x_d^{p^e}) \Ext^{h+d-2}_S(\Ext^{h+1}_S(R/J_1^{(mp^e+1)},S),S)=0
\]
for every $e\in \N$. By $(2)$ of Proposition~\ref{lcb and ses} 
\[
\lcb_2(\im(\partial_{h+d-2}^*))=\lcb_1(\Ext^{h+d-2}_S(\Ext_S^{h+1}(R/J_1^{(mp^e+1)},S),S)).
\]
As $\underline{x}^{p^e}$ annihilates $\Ext^{h+d-2}_S(\Ext_S^{h+1}(R/J_1^{(mp^e+1)},S),S))$, 
\[
\lcb_1(\Ext^{h+d-2}_S(\Ext_S^{h+1}(R/J_1^{(mp^e+1)},S),S))\leq p^e
\]
by Lemma~\ref{lemma lcb power of elements} and Corollary~\ref{corollary 0 map of Koszul cohomology}.

Next, we consider the short exact sequence
\[
0\to \Ext^{h+d-3}_S(\Ext^{h+1}_S(R/J_1^{(mp^e+1)},S),S)\to \coker(\partial_{h+d-3}^*)\to \im(\partial_{h+d-2}^*)\to 0.
\]
We established $\lcb_2(\underline{x}; \im(\partial_{h+d-2}^*))\leq p^e$. By assumption and Lemma~\ref{lemma matlis duality and higher ext modules}
\[
(x_2^{p^e},\ldots,x_{d-1}^{p^e})\Ext^{h+d-3}_S(\Ext^{h+1}_S(R/J_1^{(mp^e+1)},S),S)=0
\] 
for every $e\in \N$. By (1) of Proposition~\ref{second proposition 0 map Koszul cohomology} and Lemma~\ref{lemma lcb power of elements} we have
\[
\lcb_2(\underline{x};\coker(\partial_{h+d-3}^*))\leq p^e+p^e=2p^e.
\]
Next consider the short exact sequence
\[
0\to \im(\partial^*_{h+d-3})\to F_{h+d-3}^*\to \coker(\partial_{h+d-3}^*)\to 0.
\]
By $(2)$ of Proposition~\ref{lcb and ses} and knowing that $\lcb_2(\underline{x};\coker(\partial_{h+d-3}^*))\leq 2p^e$ we see that
\[
\lcb_3(\underline{x};\im(\partial^*_{h+d-3}))\leq 2p^e.
\]
Inductively, we find that 
\[
\lcb_j(\underline{x};\im(\partial^*_{h+d-j}))\leq (j-1)p^e
\]
and
\[
\lcb_j(\underline{x};\coker(\partial_{h+d-j-1}^*))\leq jp^e
\]
for each $2\leq j\leq d-2$. Now consider the short exact sequence
\[
0\to \Ext^{h+2}_S(\Ext^{h+1}_S(R/J_1^{(mp^e+1)},S),S)\to \coker(\partial_{h+2}^*)\to \im(\partial_{h+3}^*)\to 0.
\]
By induction, $\lcb_{d-3}(\im(\partial_{h+3}^*))\leq (d-4)p^e$, therefore by $(1)$ of Proposition~\ref{second proposition 0 map Koszul cohomology}
\[
\lcb_{d-3}(\coker(\partial_{h+2}^*))\leq (d-3)p^e.
\]
Now consider the short exact sequence
\[
0\to \im(\partial_{h+2}^*)\to F_{h+2}^*\to \coker(\partial_{h+2}^*)\to 0.
\]
Apply $(2)$ of Proposition~\ref{lcb and ses} to conclude $\lcb_{d-2}(\im(\partial_{h+2}^*))\leq (d-3)p^e$. Now consider one last short exact sequence:
\[
0\to \Ext^{h+1}_S(\Ext_S^{h+1}(R/J_1^{(mp^e+1)},S),S)\to \coker(\partial_{h+2}^*)\to \im(\partial_{h+2}^*)\to 0.
\]
We now utilize that $\underline{x}$ is a regular sequence on $\coker(\partial_{h+2}^*)$ and utilize (2) of Proposition~\ref{lcb and ses} to conclude that
\[
\lcb_{d-1}(\Ext^{h+1}_S(\Ext^{h+1}_S(R/J_1^{(mp^e+1)},S),S))=\lcb_{d-2}(\im(\partial_{h+2}^*))\leq (d-3)p^e.
\]
By Lemma~\ref{lemma lcb power of elements} the parameter sequence $\underline{x}^{d-3}=x_2^{d-3},\ldots, x_d^{d-3}$ on $R/J_1$ satisfies 
\[
\lcb_{d-1}(\underline{x}^{d-1};\Ext^{h+1}_S(\Ext^{h+1}_S(R/J_1^{(mp^e+1)},S),S)))\leq p^e
\]
for each $e\in \N$.
\end{proof}

\begin{theorem}\label{theorem on how to make test ideals agree 1} Let $(R,\fm,k)$ be a local normal Cohen-Macaulay domain of Krull dimension $d\geq 4$ and $\Q$-Gorenstein in codimension $2$. Let $p>0$ be a natural number.  Let $J_1\subseteq R$ be a choice of canonical ideal and $m\in \N$ such that $J_1^{(m)}$ is principal in codimension $2$. Suppose $S$ is a regular local ring mapping onto $R$, $R\cong S/P$, and $\Ht(P)=h$. Suppose that $I_1\subseteq R$ is an ideal of pure height $1$ with components disjoint from those of $J_1$, $I_1\cap J_1=x_1R$ is principal, and parameters $x_1,x_2\ldots,x_d$ with the following properties:

\begin{enumerate}
    \item $J_1R_{x_2}$ and $J^{(m)}_1R_{x_3}$ are principal in their respective localizations;
    \item For every $1\leq j\leq d-2$, the parameter ideal $(x_1,\ldots,x_{d-j+1})$ is such that
    \[
    (x_2,\ldots,x_{d-j+1})^{p^e}H^j_\fm(R/I_1^{(mp^e)})=0
    \] 
    for each $e\in\NN$. 
\end{enumerate}
Then the following hold:
\begin{enumerate}
\item For each $e\in \N$ there exists $\ell\in\NN$ such that
\[
\lcb_{d-1}(x^{\ell(d-1)}_2,x^{\ell(d-1)}_3,x^{d-1}_4,\ldots,x^{d-1}_d; R/J_1^{(mp^e+1)})\leq p^e+1;
\]
\item For each $e\in \N$ there exists $\ell\in\NN$ such that
\[
\lcb_{d-1}(x^{\ell(d-1)}_2,x^{\ell(d-1)}_3,x^{d-1}_4,\ldots,x^{d-1}_d; R/J_1^{mp^e+1})\leq p^e+2.
\]
\end{enumerate}
\end{theorem}

\begin{proof} 
For each $e\in \N$ let $C_e$ be the cokernel of
\[
R/J^{mp^e+1}\to \Ext^{h+1}_S(\Ext^{h+1}_S(R/J_1^{mp^e+1},S),S)\cong \Ext^{h+1}_S(\Ext^{h+1}_S(R/J_1^{(mp^e+1)},S),S)
\]
and consider the short exact sequences 
\[
0\to R/J_1^{(mp^e+1)}\to \Ext^{h+1}_S(\Ext^{h+1}_S(R/J_1^{mp^e+1},S),S)\to C_e\to 0,
\]
see Lemma~\ref{Inclusion lemma} for details.

By Lemma~\ref{Inclusion lemma} the module $C_e$ is $0$ when either $x_2$ or $x_3$ is inverted. Hence for each $e\in \N$ there exists an integer $\ell$ such that $(x_2^\ell,x_3^\ell)C_e=0$. Because $d\geq 4$ we have that $d-1\geq 3$ and $(3)$ of Proposition~\ref{second proposition 0 map Koszul cohomology} implies
\begin{align*}
\lcb_{d-1}(x_2^{\ell(d-1)},&x_3^{\ell(d-1)},x^{d-1}_4,\ldots,x^{d-1}_d;  R/J_1^{(mp^e+1)})\leq\\
& \lcb_{d-1}(x_2^{\ell(d-1)},x_3^{\ell(d-1)},x^{d-1}_4,\ldots,x^{d-1}_d;  \Ext^{h+1}_S(\Ext^{h+1}_S(R/J_1^{mp^e+1},S,S))+1.
\end{align*}
Statement (1) follows by Proposition~\ref{proposition technical lcb}.

To prove $(2)$ let $K_e=J_1^{(mp^e+1)}/J_1^{mp^e+1}$ and consider the short exact sequences
\[
0\to K_e\to R/J_1^{mp^e+1}\to R/J_1^{(mp^e+1)}\to 0.
\]
The module $K_e$ is $0$ when either $x_2$ or $x_3$ are inverted. Hence for each $e\in \N$ there exists an integer $\ell$ such that $(x_2^\ell,x_3^\ell)K_e=0$. By (1) of Proposition~\ref{second proposition 0 map Koszul cohomology} we have that
\[
\lcb_{d-1}(x_2^\ell,x_3^\ell,x_4,\ldots,x_d;  R/J_1^{mp^e+1})\leq
 \lcb_{d-1}(x_2^\ell,x_3^\ell,x_4,\ldots,x_d; R/J_1^{(mp^e+1)})+1\leq p^e+2.
\] 
\end{proof}

Theorem~\ref{Main Theorem 1} is a consequence of the following theorem.

\begin{theorem}
\label{Main Theorem 1 again}
Let $(R,\fm,k)$ be an excellent local Cohen-Macaulay normal domain of prime characteristic $p>0$, of Krull dimension $d\geq 4$, $I_1\subseteq R$ an anti-canonical ideal, and $E_R(k)$ an injective hull of the residue field. Suppose further that there exists an $m\in\NN$ so that $I_1^{(m)}$ is principal in codimension $2$ and for each $1\leq j\leq d-2$ there exists an ideal $\fa_j$ of height $d-j+1$ such that
\[
\fa_j^{p^e}H^{j}_\fm\left(\frac{R}{I_1^{(mp^e)}}\right)=0
\]
for every $e\in\NN$. Then $0^{*fg}_{E_R(k)}=0^*_{E_R(k)}$.
\end{theorem}

\begin{proof}

Our strategy is to employ Theorem~\ref{theorem on how to make test ideals agree 1} and then Theorem~\ref{theorem how to make test ideals agree using lcbs} to conclude $0^{*fg}_{E_R(k)}=0^*_{E_R(k)}$. But first, we change the ideals $\fa_j$, if necessary, so that there are inclusions $ \fa_{d-2}\subseteq \fa_{d-3}\subseteq \cdots \subseteq \fa_1$ and so that there exists parameter elements $x^m_1,x_2,x_3\in \fa_j$ for all $1\leq j\leq d-2$ with the property that $x_1R=I_1\cap J_1$ for some canonical ideal $J_1$ and the ideals $I_1R_{x_2}$ and $I_1^{(m)}R_{x_3}$ are principal in their respective localizations.

 The ideal $\fa_j\cap \fa_{j-1}\cap \cdots \cap \fa_{1}$ has height at least $d-j+1$. We can replace the ideal $\fa_j$ with $\fa_j\cap \fa_{j-1}\cap \cdots \cap \fa_{1}$ and may assume that 
 \[
 \fa_{d-2}\subseteq \fa_{d-3}\subseteq \cdots \subseteq \fa_1.
 \]
Start by choosing $x_1\in I_1$ a generic generator so that $x_1R=I_1\cap J_1$ and the ideals $I_1$ and $J_1$ have disjoint components. Clearly $x_1^{mp^e}$ annihilates $H^j_\fm(R/I_1^{(mp^e)})$ for every $e\in\NN$. The ideal $I_1$ is principal in codimension $1$, the ideal $I_1^{(m)}$ is principal in codimension $2$. Therefore there exists part of a system of parameters $x_2,x_3$ of $R/x_1R$ so that $I_1R_{x_2}$ and $I_1^{(m)}R_{x_3}$ are principal in their respective localizations. Moreover, we can replace $x_2$ and $x_3$ by suitable powers and can assume that there exists elements $a,b\in I_1$ so that $x_2I_1\subseteq aR\subseteq I_1$ and $x_3I_1^{(m)}\subseteq bR\subseteq I_1^{(m)}$. Therefore $x_2^{mp^e}I_1^{(mp^e)}\subseteq a^{mp^e}R\subseteq I^{(mp^e)}_1$ and $x_3^{p^e}I_1^{(mp^e)}\subseteq b^{p^e}R\subseteq I_1^{(mp^e)}$. Consider the short exact sequences
\[
0\to \frac{I_1^{(mp^e)}}{a^{mp^e}R}\to \frac{R}{a^{mp^e}R}\to \frac{R}{I_1^{(mp^e)}}\to 0
\]
and
\[
0\to \frac{I_1^{(mp^e)}}{b^{p^e}R}\to \frac{R}{b^{mp^e}R}\to \frac{R}{I_1^{(mp^e)}}\to 0.
\]
The elements $x_2^{mp^e}$ and $x_3^{p^e}$ annihilate $I_1^{(mp^e)}/a^{mp^e}R$ and $I_1^{(mp^e)}/b^{p^e}R$ respectively. Examining the resulting long exact sequences of local cohomology informs us that $x_2^{mp^e}$ and $x_3^{p^e}$ annihilate $H^j_\fm(R/I_1^{(mp^e)})$ for every $1\leq j\leq d-2$. Replace the element $x_2$ by $x_2^m$. Then $(x_2^{p^e},x_3^{p^e})$ annihilates $H^j_\fm(R/I_1^{(mp^e)})$ for every $1\leq j\leq d-2$. For each $e\in \NN$ the ideal $((\fa_j+(x^m_1,x_2,x_3))^{4{p^e}}$ is generated by elements which live in either $\fa_j^{p^e}$ or $(x^m_1,x_2,x_3)^{[p^e]}$ and therefore annihilate $H^j_\fm(R/I_1^{(mp^e)})$.  We replace $\fa_j$ by the ideal $(\fa_j+(x^m_1,x_2,x_3))^{4}$.

The ideal $\fa_j$ has height at least $d-j+1$ and $x^m_1,x_2,x_3\in \fa_j$. We can extend $x^m_1,x_2,x_3$ to a parameter sequence $x^m_1,x_2,x_3,\ldots,x_{d-j+1}$ in $\fa_j\subseteq \fa_{j-1}\subseteq \cdots \subseteq \fa_1$. By Theorem~\ref{theorem on how to make test ideals agree 1}, for each $e\in \NN$ there exists an $\ell$ so that
\[
\lcb_{d-1}(x_2^{\ell(d-1)},x_3^{\ell(d-1)},x_d^{d-1},\ldots,x_d^{d-1};R/J_1^{(mi+1)})\leq p^e+1.
\]
Therefore $0^{*fg}_{E_R(k)}=0^*_{E_R(k)}$ by Theorem~\ref{theorem how to make test ideals agree using lcbs}.
\end{proof}





\section*{Acknowledgement}
We are thankful to an anonymous referee for carefully reading our article and providing feedback which improved the article's exposition and content.

\bibliographystyle{alpha}
\bibliography{References}

\newcommand{\etalchar}[1]{$^{#1}$}
\begin{thebibliography}{BMP{\etalchar{+}}20}

\bibitem[Abe02]{Aberbach2002}
Ian~M. Aberbach.
\newblock Some conditions for the equivalence of weak and strong {$F$}-regularity.
\newblock {\em Comm. Algebra}, 30(4):1635--1651, 2002.

\bibitem[AE03]{AberbachEnescuTestIdeals}
Ian~M. Aberbach and Florian Enescu.
\newblock Test ideals and base change problems in tight closure theory.
\newblock {\em Trans. Amer. Math. Soc.}, 355(2):619--636, 2003.

\bibitem[AP22]{AberbachPolstraJOA}
Ian Aberbach and Thomas Polstra.
\newblock Local cohomology bounds and the weak implies strong conjecture in dimension 4.
\newblock {\em J. Algebra}, 605:37--57, 2022.

\bibitem[BH93]{BrunsHerzog}
Winfried Bruns and J\"{u}rgen Herzog.
\newblock {\em Cohen-{M}acaulay rings}, volume~39 of {\em Cambridge Studies in Advanced Mathematics}.
\newblock Cambridge University Press, Cambridge, 1993.

\bibitem[BM10]{BrennerMonsky}
Holger Brenner and Paul Monsky.
\newblock Tight closure does not commute with localization.
\newblock {\em Ann. of Math. (2)}, 171(1):571--588, 2010.

\bibitem[BMP{\etalchar{+}}20]{MixedCharMMPPaper}
Bhargav Bhatt, Linquan Ma, Zsolt Patakfalvi, Karl Schwede, Kevin Tucker, Joe Waldron, and Jakub Witaszek.
\newblock Globally +-regular varieties and the minimal model program for threefolds in mixed characteristic, 2020.

\bibitem[Bro79]{Brodmann}
M.~Brodmann.
\newblock Asymptotic stability of {${\rm Ass}(M/I\sp{n}M)$}.
\newblock {\em Proc. Amer. Math. Soc.}, 74(1):16--18, 1979.

\bibitem[BS13]{BrodmannSharpBook}
M.~P. Brodmann and R.~Y. Sharp.
\newblock {\em Local cohomology}, volume 136 of {\em Cambridge Studies in Advanced Mathematics}.
\newblock Cambridge University Press, Cambridge, second edition, 2013.
\newblock An algebraic introduction with geometric applications.

\bibitem[CEMS18]{ChiecchioEnescuMillerSchwede}
Alberto Chiecchio, Florian Enescu, Lance~Edward Miller, and Karl Schwede.
\newblock Test ideals in rings with finitely generated anti-canonical algebras.
\newblock {\em J. Inst. Math. Jussieu}, 17(1):171--206, 2018.

\bibitem[CHS10]{CutkoskyHerzogSrinivasan}
Steven~Dale Cutkosky, J\"{u}rgen Herzog, and Hema Srinivasan.
\newblock Asymptotic growth of algebras associated to powers of ideals.
\newblock {\em Math. Proc. Cambridge Philos. Soc.}, 148(1):55--72, 2010.

\bibitem[HH90]{HHJAMS}
Melvin Hochster and Craig Huneke.
\newblock Tight closure, invariant theory, and the {B}rian\c{c}on-{S}koda theorem.
\newblock {\em J. Amer. Math. Soc.}, 3(1):31--116, 1990.

\bibitem[HH94]{HHJAG}
Melvin Hochster and Craig Huneke.
\newblock Tight closure of parameter ideals and splitting in module-finite extensions.
\newblock {\em J. Algebraic Geom.}, 3(4):599--670, 1994.

\bibitem[Hoc77]{HochsterPurity}
Melvin Hochster.
\newblock Cyclic purity versus purity in excellent {N}oetherian rings.
\newblock {\em Trans. Amer. Math. Soc.}, 231(2):463--488, 1977.

\bibitem[HS15]{HunekeSmrinov}
Craig Huneke and Ilya Smirnov.
\newblock Prime filtrations of the powers of an ideal.
\newblock {\em Bull. Lond. Math. Soc.}, 47(4):585--592, 2015.

\bibitem[Hun82]{HunekeAssociatedGraded}
Craig Huneke.
\newblock On the associated graded ring of an ideal.
\newblock {\em Illinois J. Math.}, 26(1):121--137, 1982.

\bibitem[HW02]{HaraWatanabe}
Nobuo Hara and Kei-Ichi Watanabe.
\newblock F-regular and {F}-pure rings vs. log terminal and log canonical singularities.
\newblock {\em J. Algebraic Geom.}, 11(2):363--392, 2002.

\bibitem[Lip69]{LipmanRationalSingularities}
Joseph Lipman.
\newblock Rational singularities, with applications to algebraic surfaces and unique factorization.
\newblock {\em Inst. Hautes \'{E}tudes Sci. Publ. Math.}, (36):195--279, 1969.

\bibitem[Lip78]{LipmanResolution}
Joseph Lipman.
\newblock Desingularization of two-dimensional schemes.
\newblock {\em Ann. of Math. (2)}, 107(1):151--207, 1978.

\bibitem[LS99]{LyubeznikSmith}
Gennady Lyubeznik and Karen~E. Smith.
\newblock Strong and weak {$F$}-regularity are equivalent for graded rings.
\newblock {\em Amer. J. Math.}, 121(6):1279--1290, 1999.

\bibitem[LS01]{LyubeznikSmithTestIdeal}
Gennady Lyubeznik and Karen~E. Smith.
\newblock On the commutation of the test ideal with localization and completion.
\newblock {\em Trans. Amer. Math. Soc.}, 353(8):3149--3180, 2001.

\bibitem[Mac96]{Maccrimmon}
Brian~Cameron Maccrimmon.
\newblock {\em Strong {F}-regularity and boundedness questions in tight closure}.
\newblock ProQuest LLC, Ann Arbor, MI, 1996.
\newblock Thesis (Ph.D.)--University of Michigan.

\bibitem[Pol22]{PolstraMCM}
Thomas Polstra.
\newblock A theorem about maximal {C}ohen-{M}acaulay modules.
\newblock {\em Int. Math. Res. Not. IMRN}, (3):2086--2094, 2022.

\bibitem[PT18]{PolstraTucker}
Thomas Polstra and Kevin Tucker.
\newblock {$F$}-signature and {H}ilbert-{K}unz multiplicity: a combined approach and comparison.
\newblock {\em Algebra Number Theory}, 12(1):61--97, 2018.

\bibitem[SH06]{SwansonHuneke}
Irena Swanson and Craig Huneke.
\newblock {\em Integral closure of ideals, rings, and modules}, volume 336 of {\em London Mathematical Society Lecture Note Series}.
\newblock Cambridge University Press, Cambridge, 2006.

\bibitem[Smi93]{SmithThesis}
Karen~Ellen Smith.
\newblock {\em Tight closure of parameter ideals and {F}-rationality}.
\newblock ProQuest LLC, Ann Arbor, MI, 1993.
\newblock Thesis (Ph.D.)--University of Michigan.

\bibitem[Smi97]{SmithRationalSingularities}
Karen~E. Smith.
\newblock {$F$}-rational rings have rational singularities.
\newblock {\em Amer. J. Math.}, 119(1):159--180, 1997.

\bibitem[Tak04]{Takagi}
Shunsuke Takagi.
\newblock An interpretation of multiplier ideals via tight closure.
\newblock {\em J. Algebraic Geom.}, 13(2):393--415, 2004.

\bibitem[Wat94]{WatanabeInfinite}
Keiichi Watanabe.
\newblock Infinite cyclic covers of strongly {$F$}-regular rings.
\newblock In {\em Commutative algebra: syzygies, multiplicities, and birational algebra ({S}outh {H}adley, {MA}, 1992)}, volume 159 of {\em Contemp. Math.}, pages 423--432. Amer. Math. Soc., Providence, RI, 1994.

\bibitem[Wil95]{Williams}
Lori~J. Williams.
\newblock Uniform stability of kernels of {K}oszul cohomology indexed by the {F}robenius endomorphism.
\newblock {\em J. Algebra}, 172(3):721--743, 1995.

\end{thebibliography}

\end{document}